\documentclass[11pt,reqno]{amsart}
\usepackage{amsmath,amsfonts,amssymb,amscd,amsthm,amsbsy,bbm, epsf,calc, comment, caption, enumerate,mathtools}
\usepackage{color}
\usepackage[dvipsnames]{xcolor}
\usepackage{datetime}
\usepackage{soul}
\usepackage{esint}
\usepackage{enumitem}
\usepackage[utf8]{inputenc}
\usepackage[T1]{fontenc}

\usepackage[backend=biber,safeinputenc,style=numeric,giveninits=true,isbn=false,doi=false,eprint=true,url=false]{biblatex}

\AtEveryBibitem{
  \clearfield{language}
  \clearlist{language}
}

\usepackage[pdftex]{graphicx}
\usepackage{wrapfig}
\usepackage{geometry}
\usepackage[colorlinks=true, urlcolor=blue, linkcolor=blue, citecolor=green]{hyperref}

\numberwithin{equation}{section}
\newcounter{hours}\newcounter{minutes}

\theoremstyle{plain}
\newtheorem{theorem}{Theorem}[section]
\newtheorem{definition}[theorem]{Definition} 
\newtheorem{lemma}[theorem]{Lemma} 
\newtheorem{proposition}[theorem]{Proposition} 
\newtheorem{remark}[theorem]{Remark}
\newtheorem{corollary}[theorem]{Corollary} 
\newtheorem{example}[theorem]{Example}

\newtheorem*{prop*}{Proposition}

\makeatletter
\newcommand\RedeclareMathOperator{%
  \@ifstar{\def\rmo@s{m}\rmo@redeclare}{\def\rmo@s{o}\rmo@redeclare}%
}
\newcommand\rmo@redeclare[2]{%
  \begingroup \escapechar\m@ne\xdef\@gtempa{{\string#1}}\endgroup
  \expandafter\@ifundefined\@gtempa
     {\@latex@error{\noexpand#1undefined}\@ehc}%
     \relax
  \expandafter\rmo@declmathop\rmo@s{#1}{#2}}
\newcommand\rmo@declmathop[3]{%
  \DeclareRobustCommand{#2}{\qopname\newmcodes@#1{#3}}%
}
\@onlypreamble\RedeclareMathOperator
\makeatother

\def\al{\alpha}

\def\om{\omega}

\def\Om{\Omega}

\def\lam{\lambda}
\def\Lam{\Lambda}

\def\real{\mathbb R}

\def\K{\mathcal K}

\def\M{\mathcal M}

\def\R{\mathbb R}
\def\Rd{{\mathbb R}^d}

\def\S{\mathcal S}

\def\loc{\textnormal{loc}}

\def\Indicator{{\mathbbm{1}}}

\def\supp{\textnormal{supp}}

\def\intd{\int_{\Rd}}

\def\d{\mathrm{d}}

\def\and{\quad\text{and}\quad}

\RedeclareMathOperator{\div}{\textnormal{div}}

\def\polhk#1{\setbox0=\hbox{#1}{\ooalign{\hidewidth
	    \lower1.5ex\hbox{`}\hidewidth\crcr\unhbox0}}}

\newcommand{\abs}[1]{\left| #1 \right|}

\newcommand{\norm}[1]{\lVert#1\rVert}

\begin{document}

\raggedbottom

\title[Partial Hölder regularity for nonlocal equations with $L^1$-kernels]{Partial Hölder regularity for fully nonlinear nonlocal parabolic equations with integrable kernels}

\author{Minhyun Kim}
\author{Luke Schleef}
\author{Russell W. Schwab}

\address{Department of Mathematics \& Research Institute for Natural Sciences \\
Hanyang University\\
222 Wangsimni-ro\\
Seoul, South Korea}
\email{minhyun@hanyang.ac.kr}
\address{Fakultät für Mathematik\\
Universität Bielefeld\\
Universitätsstraße 25\\
33615 Bielefeld, Germany}
\email{luke.schleef@uni-bielefeld.de}
\address{Department of Mathematics\\
Michigan State University\\
619 Red Cedar Road\\
East Lansing, MI 48824}
\email{rschwab@math.msu.edu}

\begin{abstract}

In this work, we consider solutions to (fully nonlinear) parabolic integro-differential equations with integrable interaction kernels.  A typical equation would be that obtained by starting with, for $s\in(0,1)$, the $s$-fractional heat equation, but replacing the interaction kernel in the integro-differential term with one which has been truncated, for $\rho>0$, at the value $\rho^{-d-2s}$, hence integrable.  We show that solutions to these equations have a partial regularity estimate which captures differences of the solution up to the scale at which the kernel has a truncation in its singularity.  The estimates we provide are robust with respect to the truncation parameter, and they include the existing results for the original operators without truncation.  There are some earlier results for linear and elliptic cases of this situation of integrable interaction kernels, and so our work is a generalization of those to the nonlinear and parabolic setting.

\end{abstract}

\thanks{M. Kim acknowledges support from the National Research Foundation of Korea (NRF) grant funded by the Korean government (MSIT) (RS-2023-00252297).}
\thanks{R. Schwab acknowledges support from the Simons Foundation for a Travel Support for Mathematicians grant.}
\keywords{Integro-differential equations; Krylov--Safonov; Harnack inequality; Hölder regularity}
\subjclass[2020]{
45K05, 
35K55, 
35D40, 
35B65
}

\maketitle


\section{Introduction}

In this work, we are interested in establishing results that demonstrate Hölder regularity, up to an appropriate scale, for functions, say $u\colon\Rd\times(-\infty,0]\to\R$, that solve a parabolic integro-differential equation in non-divergence form.  In the simplest setting,  these equations have a structure like
\begin{align}\label{eqIntro:SimpleParabolicEq}
	\partial_t u(x,t)-Lu(x,t)=f(x,t) \quad \text{in}\enspace \Om\times I,
\end{align}
where $L$ is an integro-differential operator of the form, for $K(\cdot,\cdot,t)\colon\real^d\times\real^d\to [0,\infty)$,
\begin{align}\label{eqIntro:Operator}
Lu(x,t)=\text{p.v.}\intd(u(x+y,t)-u(x,t))K(x,y,t)\d y,
\end{align} 
$I$ is some time interval, $\Om$ is some open subset of $\real^d$, and $u$ is a bounded function that satisfies this equation in an appropriate (possibly non-classical) sense.  Our main goal is to track how the behavior of $K(x,y,t)$ for $|y|\to 0$ imparts partial Hölder regularity on a solution, up to an appropriate scale, where for some parameter, $\rho$, related to $K$,
\begin{align*}
d((x,t),(y,\tau))\geq\rho,\enspace\text{but then}\enspace\abs{u(x,t)-u(y,\tau)}\leq Cd((x,t),(y,\tau))^\al,
\end{align*}
with a constant that does not depend on $\rho$ and $d$ being the parabolic distance we introduce below.   

For the sake of accessible presentation, we state a special case of the main result of our paper as it pertains to a simple setting for linear equations with easily stated assumptions as Theorem \ref{thmIntro:InformalMainResult}, below.  However, the results we prove apply to fully nonlinear versions of \eqref{eqIntro:SimpleParabolicEq}, and the precise versions of our main results appear in Sections \ref{sec:Harnack} and \ref{sec:Hoelder}, below. The key point is that even when $K(x,\cdot,t)\in L^1(\real^d)$, if the values of $K(x,y,t)$ are large enough for $|y|\leq \rho$, then solutions will gain some regularity up to scale $\rho$.  When the kernels, $K$, are allowed to have a singularity for $|y|\to 0$, this is a very well studied family of results, which can fairly be called Krylov--Safonov theory for integro-differential equations.  There is a large collection of works for these Hölder continuity results for solutions of equations like \eqref{eqIntro:SimpleParabolicEq}, and we mention some in Section \ref{subsec:Background}.

A famous and canonical example of an integro-differential operator is the fractional Laplacian $(-\Delta)^s$, formally given by
\begin{align}\label{eqIntro:FracLaplace}
(-\Delta)^su(x)=\text{p.v.}\intd\frac{u(x)-u(x+y)}{|y|^{d+2s}}\d y\coloneqq\lim_{\varepsilon\to 0^+}\int_{\Rd\setminus B_\varepsilon}\frac{u(x)-u(x+y)}{|y|^{d+2s}}\d y, \quad x\in\Rd,
\end{align}
for $s\in(0,1)$, where $\text{p.v.}$ denotes the Cauchy principal value of the integral. In this integral, we integrate the difference of $u$ at $x$ and any other point in space with respect to some weight, given by the kernel
\begin{align}\label{eqIntro:FracLaplaceKernel}
K(y)=\frac{1}{|y|^{d+2s}}, \quad y\in\Rd\setminus\{0\}.
\end{align}
This kernel has a singularity at the origin, which causes regularizing effects of the operator $(-\Delta)^s$. For that reason, operators like \eqref{eqIntro:FracLaplace} are called integro-differential operators of order $2s$. In this work, we will also consider integrable kernels that do not need to have such singularities. A prime example one should keep in mind is the truncated version of $(-\Delta)^s$, it has the kernel
\begin{align}\label{eqIntro:KRho}
K_\rho(y)=\min\left\{\frac{1}{|y|^{d+2s}},\frac{1}{\rho^{d+2s}}\right\}, \quad y\in\Rd
\end{align}
for some given $\rho>0$. Outside a ball of radius $\rho$ this kernel coincides with \eqref{eqIntro:FracLaplaceKernel}, whereas it is constant inside this ball. It is clear that one would not expect that the full regularization properties remain for those operators. However, we will provide partial Hölder estimates for viscosity solutions to parabolic equations involving such operators, and moreover, they will be robust in the parameter $\rho$.  Precise assumptions about the kernels, $K$, the general version of \eqref{eqIntro:SimpleParabolicEq}, and viscosity solutions are given in Section \ref{subsec:Pre}.

\begin{theorem}[Simplest Setting]\label{thmIntro:InformalMainResult}
Let $\rho>0$ and assume that $K$ satisfies, for some $0<\lam<\Lam$,
\begin{align}\label{eqIntro:EllipticityPointwise}
	\lam \min\left\{ \abs{y}^{-d-2s}, \rho^{-d-2s}   \right\} 
	\leq K(x,y,t)
	\leq \Lam\abs{y}^{-d-2s},
\end{align}
with also $K(x,y,t)=K(x,-y,t)$ if $s=1/2$. Let $R>\rho$ and $f\in L^\infty(B_R\times I_R)$, where $I_R = (-R^{2s},0]$. Assume that $u\in L^\infty(\Rd\times I_R)$, $Lu$ is as in \eqref{eqIntro:Operator}, and $u$ is a viscosity solution of
\begin{align*}
\partial_tu-Lu=f \enspace\text{in}\enspace B_R\times I_R.
\end{align*}
Then there exist universal parameters $\alpha\in(0,2s)$ and $C>0$ such that
\begin{align*}
\sup_{\substack{(x,t),(y,\tau)\in B_{R/2}\times I_{R/2}\\d((x,t),(y,\tau))>\rho}}\frac{|u(x,t)-u(y,\tau)|}{d((x,t),(y,\tau))^\alpha}\leq\frac{C}{R^\alpha}\left(\|u\|_{L^\infty(\Rd\times I_R)}+R^{2s}\|f\|_{L^\infty(B_R\times I_R)}\right).
\end{align*}
\end{theorem}

The general version of Theorem \ref{thmIntro:InformalMainResult} appears as Theorem \ref{thmHoelder:PartialHoelderAt0} and Corollary \ref{corHoelder:CorollaryHoelder}, which is an application of the weak Harnack inequality of Section \ref{sec:Harnack}. Theorem \ref{thmHoelder:PartialHoelderAt0} applies to operators for which the corresponding kernel need not be symmetric in $y$, allows for drift terms if $s\geq 1/2$, and applies to a class of fully nonlinear equations related to \eqref{eqIntro:EllipticityPointwise}.

The phrase \emph{universal} means, in particular, that $\alpha$ and $C$ do not depend on $\rho$, such that the estimate is robust as $\rho\to 0^+$, and our result recovers the one in \cite{Sil14}. We include proofs that apply directly to both the truncated setting and the original theory, when no truncation is present in the lower bound on $K$. To be precise, such an assumption can be written for $y\not=0$ as
\begin{align*} 
	K(x,y,t)\geq \lam\min\{ \abs{y}^{-d-2s}, \rho^{-d-2s} \}\enspace\text{if}\ \rho>0,\enspace\text{and}\enspace K(x,y,t)\geq \lam\abs{y}^{-d-2s} \enspace\text{if}\ \rho=0.
\end{align*}
It will appear again with the other assumptions in Section \ref{subsec:Pre}.

We want to immediately address the two most pertinent remarks about our theorems (which appear in Sections \ref{sec:Harnack} and \ref{sec:Hoelder}) with regards to previous work.

\begin{remark}
	Theorems \ref{thmIntro:InformalMainResult}, \ref{thmHarnack:WeakHarnackDrift}, \ref{thmHarnack:WeakHarnackNoDrift}, \ref{thmHoelder:PartialHoelderAt0} are all modeled after those that have appeared in \cite{Sil14}, which corresponds to $\rho=0$.
	The fact that our estimates are robust as $\rho\to0^+$ (and in fact allow for $\rho=0$) means that our results generalize those of \cite{Sil14} to the case when $K$ may be integrable. We capture the same results presented in \cite{Sil14} by setting $\rho=0$ here and in the main results, below.  There is a technical issue which is required to be overcome in this generalization, and it has to do with the fact that when $\rho=0$, the relevant extremal equations are scale invariant, whereas when $\rho>0$, they are not.
\end{remark}

\begin{remark}
	There is a precedent for such a partial Hölder regularity result for integrable kernels in \cite{FelmerDosPrazeresTopp-2018InteriorRegZeroOrderOpsFracLaplace-IsrealJournMath}.  In that work, they study (effectively) the corresponding Poisson equation for the kernel, $K_\rho$ from \eqref{eqIntro:KRho}, directly.  They obtain the same regularity result as presented in Theorem \ref{thmIntro:InformalMainResult}, above, for that kernel.  Thus, our results could be considered a generalization of \cite{FelmerDosPrazeresTopp-2018InteriorRegZeroOrderOpsFracLaplace-IsrealJournMath} to the setting that includes fully nonlinear and parabolic equations.
\end{remark}

\begin{remark}
	We are uncertain if a result like Theorem \ref{thmIntro:InformalMainResult} would hold for weak solutions of divergence form parabolic equations.  All of the methods we know for such situations are modeled on nonlocal adaptations of either De Giorgi's or Moser's techniques.  Both of those rely heavily on Sobolev and Poincar\'e inequalities (or similar variants), and it is not clear that there is an appropriate modification of those results to a situation that involves an integrable interaction kernel in the relevant energy space.  We note that in bounded domains, the spaces $L^2(\Omega)$ and
\begin{align*}
H_\rho^s(\Omega)\coloneqq\left\{u\in L^2(\Omega) \ \Big\vert\  \int_\Omega\int_\Omega\frac{(u(x)-u(y))^2}{\max\{|x-y|^{d+2s},\rho^{d+2s}\}}\d x\d y<\infty\right\}
\end{align*}
coincide. Hence, an arbitrary $H_\rho^s$-function cannot be expected to lie in some $L^{2+\varepsilon}$-space leading to a lack of embedding. Thus, if some result is used, it must be a modification of the typical embedding to account for the fact that although the interaction kernel is integrable, it has a very large mass near the diagonal. We do not know of any such result.
\end{remark}

\begin{remark}
The results in this work are not robust as $s\to 1^-$, which means they do not recover certain results for second-order parabolic equations in non-divergence form. In fact, the weak Harnack inequality we will present, in general fails for those operators and relies on the integral structure of the operators, like those in \eqref{eqIntro:Operator}. However, it would be interesting to adapt our approach to a robust setting as in \cite{CaSi-09RegularityIntegroDiff} or \cite{SS16} to investigate the local equivalent of truncation of integral kernels.  (We note that the results in \cite{CaSi-09RegularityIntegroDiff} and \cite{SS16}, among \emph{many} others contain the previously known results for local second order equations.)
\end{remark}

\subsection*{Strategy of the proof}
The strategy of the proof follows the approach of \cite{Sil14}, which exploits the fact that the nonlocality of the equation can be used to more directly obtain a relationship between the pointwise values of a solution and its integral, in contrast to the usual approach for second order equations. More precisely, we will prove a weak Harnack inequality for supersolutions, which allows us to provide a uniform lower bound for solutions on different scales. This will lead to a partial Hölder regularity in accordance with the Krylov--Safonov theory in the local case, see \cite{KS81}. The main difficulty in our setting is the absence of scaling invariance of our kernels. Hence, we need to deduce the weak Harnack inequality not only on unit scale, but on any sufficiently large scale. This has to be done carefully in order to get estimates independent of the size of the domain and independent of the truncation parameter. For that, we need additional auxiliary results to study the properties of the operators, which were not needed in \cite{Sil14}.

\subsection*{Organization of the paper}
The rest of this work is organized as follows. In Section \ref{subsec:Background} we list and discuss some earlier and related results. In Section \ref{subsec:Pre} we agree on some notations and introduce the admissible operators and equations we study in this framework. Moreover, we introduce the concept of viscosity solutions. Section \ref{sec:Harnack} will be the main part, in which we prove a weak Harnack inequality for supersolutions. Afterwards, in Section \ref{sec:Hoelder} we iterate the weak Harnack inequality to obtain partial Hölder regularity. Furthermore, we obtain a proof for Liouville's theorem. In Section \ref{subsec:Approximation} we discuss applications of our results, in particular, how one can approximate solutions to integro-differential equations by solutions obtained by integrable kernels. In the end, we discuss the special case of an elliptic equation in Section \ref{subsec:Elliptic}, where solutions can be defined pointwise.


\subsection{Background and related results}\label{subsec:Background}

There are now many works that address Krylov--Safonov results for linear and fully nonlinear equations in \emph{non-divergence form}, similar to \eqref{eqIntro:SimpleParabolicEq}.  We point out that the framework and methods for this family of results is distinctly different that those corresponding to the divergence form equations, for results like De Giorgi--Nash--Moser.  A list of current results for divergence equations appears in the recent work \cite{KassmannWeidner-2024ParabolicHarnack-Duke}.  Here, we focus on those for non-divergence equations.  For the case of $\rho=0$, the collection of works is quite extensive, and for the case of $\rho>0$, to the best of our knowledge, there are far fewer results.

We want to point out that the type of results related to Theorem \ref{thmIntro:InformalMainResult} are important for many reasons, but two noteworthy features are: they give a uniform Hölder modulus for solutions depending on very rough features of the equation and the solution; and they offer a universal \emph{gain of regularity}, i.e.\ solutions are typically only a priori continuous, and the Krylov--Safonov result says they indeed have a bit more regularity than expected.  This is, of course, well known, classical, and widely used for the case when $\rho=0$, but it is more interesting for $\rho>0$.

\subsubsection*{Results for kernels with a singularity of size $\abs{y}^{-d-2s}$}
When $\rho=0$, with a properly chosen normalization constant, second order parabolic equations are a special limiting case of equations like \eqref{eqIntro:SimpleParabolicEq}, and thus for the non-divergence setting, results like Theorem \ref{thmIntro:InformalMainResult} are reasonably classified as Krylov--Safonov results.  We importantly note, that this work does not include the correct normalization constant (multiplying $K$ by $(1-s)$) to achieve the local theory as a sub-case.  

The program of extending the Hölder regularity to non-divergence integro-differential equations started with elliptic equations, and goes back at least to \cite{BaLe-2002Harnack} and \cite{BaKa-05Holder}, which took the probabilistic approach to the theory (as in the original by Krylov and Safonov \cite{KS81}).  This was followed by \cite{Silv2006-HolderIntDiff-IUMJ} with an analytic approach.  The first result that is robust as $s\to1^-$, i.e.\ captures the second order theory, is \cite{CaSi-09RegularityIntegroDiff}, which was also for elliptic equations. This was followed by quite a few robust results for both elliptic and parabolic equations, extending the family of admissible kernels, including: \cite{Chan-2012NonlocalDriftArxiv}, \cite{ChangLaraDavila-2016HolderNonlocalParabolicDriftJDE}, \cite{GuSc-12ABParma}, \cite{KassRangSchwa-2014RegularityDirectionalINDIANA}, \cite{SS16}.  A noteworthy modification of the previous results is to consider the behavior not as $s\to 1^-$ but rather as $s\to 0^+$, and \cite{KassmannMimica-2013IntrinsicScalingJEMS} investigates this where the singularity of the kernel can be very weak, but still non-integrable. It turns out that for results that are not robust, some interesting features of the equation can be exploited to obtain a remarkably straightforward version of the weak Harnack inequality for supersolutions (as in Theorem \ref{thmHarnack:WeakHarnackDrift}, below).  This seems to have originated in \cite{Sil14}, and as we mentioned earlier, this is the result on which we have modeled our approach.

\subsubsection*{Results for integrable kernels}
There are many ways one may arrive at an equation like \eqref{eqIntro:SimpleParabolicEq} with an integrable interaction kernel. If the starting point is the analysis of generators of Markov processes, then the results of Courr\`ege \cite{Courrege-1965formePrincipeMaximum} (or the L\'evy--Khinchin formula for the translation invariant setting) indicate that any kernel, $K$, that satisfies
\begin{align*}
	\sup_{x,t}\int_{\real^d} \min\{\abs{y}^2,1\} K(x,y,t)dy <\infty,
\end{align*}
would arise naturally, that is to say, both integrable and non-integrable kernels can appear, as long as the singularity can be integrated against $\abs{y}^2$ near $y=0$. A couple of works using integrable kernels in the probabilistic setting are \cite{Cha07}, \cite{ChasseigneChavesRossi-2006AsymptoticNonlocalDiffusion-JMPA}. However, a completely different origin of \eqref{eqIntro:SimpleParabolicEq} is the study of continuum limits for interacting particles systems and phase transitions, and some representative works are \cite{BatesChmaj-1999IntDiffModelPhaseTransitions-JStatPhys}, \cite{BaFiReXi-97}.  Of course, another natural reason to consider integrable kernels for \eqref{eqIntro:SimpleParabolicEq} is to give an approximation of equations for $\rho=0$ by those for which the interaction kernel is more tame.   Many of the results for integrable interaction kernels refer to the corresponding equations as nonlocal diffusion equations. Existence, uniqueness and asymptotic behavior for various problems are studied in \cite{AMRT10}.  Regarding regularity results, there is a series of works for integrable kernels in \cite{FelmerTopp-2013ConvergenceZeroOrderToNonlinearFracHeat-IsraelJournMath}, \cite{FelmerTopp-2015UnifEquicontinuityZeroOrderEqApproxFracLaplace-CPDE}, \cite{FelmerDosPrazeresTopp-2018InteriorRegZeroOrderOpsFracLaplace-IsrealJournMath}. Among other interesting features, \cite{FelmerTopp-2013ConvergenceZeroOrderToNonlinearFracHeat-IsraelJournMath} and \cite{FelmerTopp-2015UnifEquicontinuityZeroOrderEqApproxFracLaplace-CPDE} demonstrate how regularity of solutions to equations with integrable kernels can be inherited from either the initial data of a Cauchy problem or the right hand side of the equation. They demonstrate partial regularity for solutions, but not a gain of regularity as in the Krylov--Safonov type results. The result in \cite{FelmerDosPrazeresTopp-2018InteriorRegZeroOrderOpsFracLaplace-IsrealJournMath} does demonstrate a gain in regularity, and they obtain a partial regularity result, like ours, for an integrable version of the linear and elliptic Poisson equation, corresponding to $K_\rho$, defined in \eqref{eqIntro:KRho}.


\subsection{Preliminaries}\label{subsec:Pre}

Here we give more precise details and formulation for what kind of operators can appear in \eqref{eqIntro:Operator}, and what it means for functions to solve equations like \eqref{eqIntro:SimpleParabolicEq}.  We also fix some notation for our presentation.

\subsubsection{Sets and functions in $\real^d\times(-\infty,0]$}
 For $R>0$ and $x\in\Rd$, we will denote
\begin{align*}
B_R(x)\coloneqq\{y\in\Rd\mid|x-y|<R\},\quad B_R\coloneqq B_R(0) \and I_R\coloneqq(-R^{2s},0],
\end{align*}
where the dimension $d\geq 1$ and the parameter $s\in(0,1)$ are fixed during this whole article. Furthermore, let us introduce the parabolic cylinders
\begin{align*}
Q_R\coloneqq B_R\times I_R.
\end{align*}
Note that these sets are open balls in $\Rd\times(-\infty,0]$ induced by the parabolic distance function
\begin{align*}
d((x,t),(y,\tau))\coloneqq\max\left\{|x-y|,|t-\tau|^\frac{1}{2s}\right\}.
\end{align*}
Finally, for a real number $a\in\R$, we denote the positive and negative part of $a$ by $a_+\coloneqq\max\{a,0\}$ and $a_-\coloneqq(-a)_+$, respectively.

\subsubsection{Assumptions on the operators}

We will treat both linear and nonlinear operators that are modeled on \eqref{eqIntro:Operator}.  For now, we will call these operators $F$.

For each fixed $t$, we think of $F$ as an operator for functions $u(\cdot,t)$ in the space variable, 
and the most basic requirement is that
\begin{align*}
	F\colon C_\loc^{1,1}(\Rd)\cap L^\infty(\Rd)\to C(\Rd),
\end{align*}
whereby we would say $F(u)(\cdot,t)\in C(\real^d)$ and the value at $x$ is $F(u)(x,t)\in\real$. In fact, the domain of definition of $F$ can be relaxed considerably, when the particular value of $s$ is taken into consideration, but for now, $C^{1,1}$ will suffice.  As we explain next, we take $F$ to have a more specific formulation in terms of integration kernels, $K(x,y,t)$ as above, and the $t$-dependence should become more clear.

The precise assumptions on the interaction kernels are stated as follows. Note that we cover the known case where $\rho=0$ and the case where $\rho>0$ simultaneously. For the sake of readability, we use the convention that, if $\rho=0$, then $\rho^{-d-2s}=\infty$, i.e.
\begin{align*}
\min\{|y|^{-d-2s},\rho^{-d-2s}\}=|y|^{-d-2s},
\end{align*}
throughout this work. We fix some parameters $\lambda,\Lambda>0$ and $\rho\geq 0$, and we denote by $\K_\rho$ the collection of all kernels $K\colon\Rd\to[0,\infty)$ that satisfy
\begin{align}\label{eqPre:LowerEllipticityAssumption}
K(y)\geq\lambda\min\left\{|y|^{-d-2s},\rho^{-d-2s}\right\} \quad\text{for all}\ y\in\Rd\setminus\{0\}
\end{align}
and
\begin{align}\label{eqPre:UpperEllipticityAssumption}
\int_{B_{2r}\setminus B_r}K(y)\d y\leq\Lambda r^{-2s} \quad\text{for all}\ r>0.
\end{align}
In the case $s=1/2$, we additionally assume that
\begin{align}\label{eqPre:SymmetryAssumption}
\int_{B_{2r}\setminus B_r}yK(y)\d y=0 \quad\text{for all}\ r>0.
\end{align}
Let us emphasize that, for $\rho=0$ with the above convention, \eqref{eqPre:LowerEllipticityAssumption} becomes the usual ellipticity condition
\begin{align*}
K(y)\geq\frac{\lambda}{|y|^{d+2s}} \quad\text{for all}\ y\in\Rd\setminus\{0\}.
\end{align*}
The assumptions \eqref{eqPre:UpperEllipticityAssumption} and \eqref{eqPre:SymmetryAssumption} coincide with those of \cite{SS16}, but here, for simplicity, we take a global lower bound on $K(y)$ in \eqref{eqPre:LowerEllipticityAssumption}, which is much less general than the corresponding assumption in \cite{SS16}. Note that symmetric kernels, $K(y)=K(-y)$, satisfy \eqref{eqPre:SymmetryAssumption}. However, \eqref{eqPre:SymmetryAssumption} is more general and only necessary for the critical case $s=1/2$.

For $y\in\Rd$, we define the difference operator $\delta_y$ by
\begin{align}\label{eqPre:DefinitionDifferences}
\delta_yu(x)=\begin{cases}
	u(x+y)-u(x),\ &s<1/2, \\	
	u(x+y)-u(x)-y\cdot\nabla u(x)\Indicator_{B_1}(y),\ &s=1/2, \\
	u(x+y)-u(x)-y\cdot\nabla u(x),\ &s>1/2,
\end{cases}
\end{align}
and the integro-differential operators
\begin{align}\label{eqPre:Operator}
L_Ku(x)=  \intd\delta_yu(x)K(y)\d y,\quad x\in\Rd,
\end{align}
for $K\in\K_\rho$.  We note that $L_K$ can be defined classically on any function that locally has enough regularity to allow $\delta_yu$ to have enough decay to balance any possible singularity in $K$.  However, we do not need the most general setting for the domain of definition of $L_K$, and we refer to Lemma \ref{lemmaPre:PointwiseDefined}, below, for a situation which is good enough for our needs. At this point, note that for $\rho>0$, the class of admissible kernels $K\in\K_\rho$ contain many that are integrable, and in this case one can understand $L_Ku$ even if $u\in L^\infty(\Rd)$ only. For this subclass of integrable kernels, one could define $\delta_y u(x)=u(x+y)-u(x)$ and $L_Ku$ will be well defined. However, in order to get robust results for $\rho\to 0^+$, it is unavoidable to define $\delta_yu$ as in \eqref{eqPre:DefinitionDifferences}.

For some given function $f\colon Q_R\to\R$, we are interested in regularity properties of solutions $u$ to the equation
\begin{align}\label{eqPre:Equation}
\partial_tu-F(u)=f,
\end{align}
where $\partial_tu$ denotes the left-derivative of $u$ in time and $F$ is a fully nonlinear operator with a drift, which is elliptic with respect to $\K_\rho$. The simplest examples one should have in mind are the linear operators
\begin{align}\label{eqPre:LinearWithDrift}
L_{K,b}u=b\cdot\nabla u+L_Ku
\end{align}
with $L_K$ as in \eqref{eqPre:Operator} and a drift $b\colon Q_R\to\R^d$ satisfying $\|b\|_{L^\infty(Q_R)}\leq\Lambda$ or $b\equiv0$ when $s<1/2$. More specifically, $F$ should satisfy
\begin{equation}\label{eqPre:FullyNonlinear}
\begin{aligned}
\text{for}\ s\geq \frac{1}{2},\quad &\M_\rho^-(u-v)-\Lambda|\nabla(u-v)|\leq F(u)-F(v)\leq\M_\rho^+(u-v)+\Lambda|\nabla(u-v)| \\
\text{and for}\ s<\frac{1}{2},\quad &\M_\rho^-(u-v)\leq F(u)-F(v)\leq\M_\rho^+(u-v),
\end{aligned}
\end{equation}
for any $u,v\in C_\loc^{1,1}(\Rd)\cap L^\infty(\Rd)$, where $\M_\rho^\pm$ denote the Pucci extremal operators with respect to $\K_\rho$, that are,
\begin{align}\label{eqPre:DefinitionPucci}
\M_\rho^-u(x)\coloneqq\inf_{K\in\K_\rho}L_Ku(x) \quad\text{and}\quad \M_\rho^+u(x)\coloneqq\sup_{K\in\K_\rho}L_Ku(x)
\end{align}
whenever $L_Ku(x)$ is defined in \eqref{eqPre:Operator}. We note that the cases in \eqref{eqPre:FullyNonlinear} are to emphasize that when $s<1/2$, we do not allow for a drift in our equations. Canonical examples of nonlocal fully nonlinear operators are given by Isaac operators
\begin{align}\label{eqPre:Isaac}
F(u)(x,t)=\inf_\alpha\sup_\beta L_{K^{\alpha,\beta},b^{\alpha,\beta}}u(x,t),
\end{align}
where $K^{\alpha,\beta}(x,\cdot,t)\in\K_\rho$ and $\|b^{\alpha,\beta}\|_{L^\infty(Q_R)}\leq\Lambda$ for any $x,t,\alpha$ and $\beta$ (see \cite{GuSc-2019MinMaxNonlocalCALCVARPDE}, \cite{GuSc-2019MinMaxEuclideanNATMA}). It is natural to assume that $F(0)=0$, as it is the case for \eqref{eqPre:Operator}, \eqref{eqPre:LinearWithDrift} and \eqref{eqPre:Isaac}. So, \emph{excluding any drift terms}, $F(u)$, by construction, has the same domain of definition as $\M_\rho^\pm$. Let us mention that we only allow for a drift if $s\geq 1/2$ because otherwise, the drift is the dominating term in the equation, since it has order $1$. As we are investigating the regularizing properties of the nonlocal operator in space, which is of order $2s$, we avoid this case. It turns out that the case $s=1/2$ can be treated with drift.

Some further comments about the drift for $s\geq 1/2$ may be useful. Since we are considering non-symmetric kernels, there can be a drift term that is incorporated into $K$ itself.  One way to see this is to decompose any kernel into an even and odd part, i.e.\ $K=K_e+K_o$, and, in fact, $L_{K_0}$ becomes a drift operator. Conversely, given any linear operator $L_{K,b}$ as in \eqref{eqPre:LinearWithDrift}, one can change $K$ in a suitable way to obtain the same operator, but (technically) without any drift.  One reason this causes an issue is that a standard way of representing operators like those here is to canonically use $\delta_yu(x) = u(x+y)-u(x)-y\cdot\nabla u(x)\Indicator_{B_1}(y)$, regardless of the value of $s$.  Then when implementing the natural rescaling of the operator, drift terms appear to correct for the change to $y\cdot\nabla u(x)\Indicator_{B_1}(y)$.  For this reason, one may link drift and kernel as in \cite[Definition 2.3]{ChangLaraDavila-2016HolderNonlocalParabolicDriftJDE} to define the admissible class of operators.  However, this technical issue does not appear in our work, thanks to either for $s>1/2$, the definition of $\delta_y u$ which does not truncate the gradient term, or for $s=1/2$, the cancellation/symmetry assumption in \eqref{eqPre:SymmetryAssumption}. The truncation of the gradient term in \eqref{eqPre:DefinitionDifferences} cannot be avoided if $s=1/2$. One can see this come into play when the drift term is removed, as in Theorem \ref{thmHarnack:WeakHarnackNoDrift}, and also the property \eqref{eqPre:AlternativeRepresentationOperator}, below.

\subsubsection{Some properties of $K$, $L_K$, and $\M_\rho^\pm$}

Let us collect some properties of the operators that follow from the taken assumptions on the kernels. The first lemma is standard and will be presented without proof.  In fact, it only uses properties \eqref{eqPre:UpperEllipticityAssumption} and \eqref{eqPre:SymmetryAssumption} for the class $\K_\rho$.

\begin{lemma}\label{lemmaPre:PointwiseDefined}
Let $x\in\Rd$, $R>0$ and $K\in\K_\rho$, let $L_K$ be as in \eqref{eqPre:Operator} and $\M_\rho^\pm$ as in \eqref{eqPre:DefinitionPucci}. Then 
\begin{align*}
	L_Ku(x) = \int_{\real^d} \delta_y u(x)K(y)\d y
	\ \ \ \ \text{and}\ \ \ \ 
	\M_\rho^\pm u(x)
\end{align*}
are defined for any function $u\in C^{k,\alpha}(\overline{B}_R(x))\cap L^\infty(\Rd)$ such that $R>0$, $k\geq 0$, $\alpha\in(0,1]$ and $k+\alpha>2s$.
\end{lemma}

Note that this result is not sharp, just sufficient for classical evaluation.  For example, if $\rho=0$ and $K(y)=K(-y)$, then $L_K$ is well defined as soon as $u(x+y)+u(x-y)-2u(x)$ has an estimate of the form $\abs{y}^{2s}\om(\abs{y})$ where $\omega$ is a Dini-modulus.  But these issues are not pertinent to our results.

\subsubsection{Definition of Viscosity solutions and some properties}

The concept for solutions in this framework is the notion of viscosity solutions, which we will introduce now. Note that the definition of fully nonlinear operators in \eqref{eqPre:FullyNonlinear} includes $\M_\rho^\pm$, for example.

\begin{definition}\label{defPre:Viscosity}
Let $\Omega\subseteq\Rd$ be a domain, $I\subseteq(-\infty,0]$ be a left-open interval and $f\colon\Omega\times I\to\R$ be a given function. We say that a function $u\in L^\infty(\Rd\times I)$, which is upper semicontinuous in $\Omega\times I$, solves the equation
\begin{align*}
\partial_tu-F(u)\leq f \enspace\text{in}\enspace \Omega\times I
\end{align*}
in the viscosity sense if, for any $(x_0,t_0)\in\Omega\times I$, any $\varepsilon>0$ such that $B_\varepsilon(x_0)\times(t_0-\varepsilon,t_0]\subseteq\Omega\times I$ and any test function $\varphi\colon B_\varepsilon(x_0)\times(t_0-\varepsilon,t_0]\to\R$ that is left-differentiable in time, $C^{1,1}$ in space, and also satisfies $u(x_0,t_0)=\varphi(x_0,t_0)$ and $u\leq\varphi$ in $B_\varepsilon(x_0)\times(t_0-\varepsilon,t_0]$, the function
\begin{align}\label{eqPre:ViscosityDefinitionConstructionV}
v\coloneqq\begin{cases}
\varphi,\ &\text{in}\ B_\varepsilon(x_0)\times(t_0-\varepsilon,t_0], \\
u,\ &\text{elsewhere in}\ \R^d\times I
\end{cases}
\end{align}
satisfies
\begin{align*}
\partial_tv(x_0,t_0)-F(v)(x_0,t_0)\leq f(x_0,t_0).
\end{align*}
In the same way, we say that a function $u\in L^\infty(\Rd\times I)$ that is lower semicontinuous in $\Omega\times I$ solves
\begin{align*}
\partial_tu-F(u)\geq f \enspace\text{in}\enspace \Omega\times I
\end{align*}
if, for any $(x_0,t_0)\in\Omega\times I$ and any test function, $\varphi$ as above, but with $u\geq\varphi$ in $B_\varepsilon(x_0)\times(t_0-\varepsilon,t_0]$, the function $v$, constructed in \eqref{eqPre:ViscosityDefinitionConstructionV}, satisfies
\begin{align*}
\partial_tv(x_0,t_0)-F(v)(x_0,t_0)\geq f(x_0,t_0).
\end{align*}
\end{definition}

We recall that $F(v)(x_0,t_0)$ is always defined by Lemma \ref{lemmaPre:PointwiseDefined} because $v\in L^\infty(\Rd\times I)$ and in a neighborhood, $v=\varphi$ is $C^{1,1}$ with respect to space.

Note that any viscosity solution to \eqref{eqPre:Equation}, that is, both subsolution and supersolution, needs to be continuous a priori. Recall that for local differential equations, the values of the test functions only matter in a neighborhood of the evaluation point. However, due to the nonlocality, the behavior of the test function in the whole space would be important.  One point of view, which we take here, is since the purpose of evaluating with the test function is to be able to apply derivatives, we do not want to take into account values of the test function far away from the point of evaluation. In particular, for our definition, the ordering between test function and candidate solution does not need to be preserved everywhere. This is the reason to adjust the test function as done in \eqref{eqPre:ViscosityDefinitionConstructionV}. We note that there are equivalent notions of test functions that require a global ordering with the candidate solution instead of using the function in \eqref{eqPre:ViscosityDefinitionConstructionV}, which is explained in \cite[Section 1.2]{BarlesImbert-07SecondOrderIntDiffRevisit-AIHP}.   We will need the following lemma that gives more information when the test function itself preserves the ordering at any point in space. The proof follows the lines of \cite[Lemma 4.5]{Sil14} and can be found in Appendix \ref{secApp:Proofs}.

\begin{lemma}\label{lemmaPre:GlobalTestFunction}
Let $\Omega\subseteq\Rd$ be a domain, $I\subseteq(-\infty,0]$ be a left-open interval and $f\colon\Omega\times I\to\R$. Assume that $u$ satisfies
\begin{align*}
\partial_tu+\Lambda|\nabla u|-\M_\rho^-u\geq f \enspace\text{in}\enspace \Omega\times I
\end{align*}
in the viscosity sense. Let $\varphi$ be a test function for the point $(x_0,t_0)\in\Omega\times I$ as in Definition \ref{defPre:Viscosity} such that $u(\cdot,t_0)\geq\varphi(\cdot,t_0)$ in $\Rd$ and assume additionally that $\varphi(\cdot,t_0)\in L^\infty(\Rd)$. Then
\begin{multline}\label{eqPre:GlobalTestFunction}
\partial_t\varphi(x_0,t_0)+\Lambda|\nabla\varphi(x_0,t_0)|-\M_\rho^-\varphi(x_0,t_0) \\
-\intd(u(x_0+y,t_0)-\varphi(x_0+y,t_0))\min\left\{\frac{\lambda}{|y|^{d+2s}},\frac{\lambda}{\rho^{d+2s}}\right\}\d y\geq f(x_0,t_0).
\end{multline}
\end{lemma}

The analogue result for subsolutions holds as well.


\section{Weak Harnack inequality}\label{sec:Harnack}

A common way to obtain Hölder regularity is to prove a weak Harnack inequality, which means that any nonnegative supersolution in a cylinder is bounded away from zero in a portion of that cylinder, but shifted in time. Afterwards, the resulting inequality will be iterated to obtain partial Hölder regularity. The usual manner is to prove weak Harnack inequality on unit scale and rescale any function to unit scale in the iteration. However, our operators are not scaling invariant such that we prove weak Harnack inequalities on any scale. If the drift is present, our weak Harnack inequality reads as follows:

\begin{theorem}\label{thmHarnack:WeakHarnackDrift}
Assume that $s\geq 1/2$. Let $a\geq 0, R_0>0$ and $R\in(\rho,R_0]$. Moreover, assume that $u\in L^\infty(\Rd\times I_R)$ satisfies $u\geq 0$ in $\Rd\times\overline{I}_R$ and
\begin{align}\label{eqHarnack:EquationDrift}
\partial_tu+\Lambda|\nabla u|-\M_\rho^-u\geq -R^{-2s}a \enspace\text{in}\enspace Q_R
\end{align}
in the viscosity sense. Then
\begin{align}\label{eqHarnack:Result}
\inf_{Q_{R/2}}u\geq c\int_{I_R\setminus I_{R/2}}\intd\frac{u(x,t)}{(R+|x|)^{d+2s}}\d x\d t-a
\end{align}
for some constant $c>0$ that depends only on $d,s,\lambda,\Lambda$ and $R_0$.
\end{theorem}

Without a drift, the statement holds true for any value of $s$.

\begin{theorem}\label{thmHarnack:WeakHarnackNoDrift}
Let $a\geq 0$, $R>\rho$ and $u\in L^\infty(\Rd\times I_R)$ satisfy $u\geq 0$ in $\Rd\times\overline{I}_R$ and
\begin{align}\label{eqHarnack:EquationNoDrift}
\partial_tu-\M_\rho^-u\geq -R^{-2s}a \enspace\text{in}\enspace Q_R
\end{align}
in the viscosity sense. Then \eqref{eqHarnack:Result} holds for some constant $c>0$ that depends only on $d,s,\lambda$ and $\Lambda$.
\end{theorem}

Let us recall at this point that the choice $\rho=0$ is always allowed. In particular, we cover the weak Harnack inequality as in \cite[Theorem 5.1]{Sil14} on any scale. We believe that it might be useful in other contexts to have the proof on arbitrary scales even in the case $\rho=0$.

The proof follows the lines of \cite{Sil14}, but it has many additional arguments due to a lack of working at unit scale. In order to have the constant $c$ in \eqref{eqHarnack:Result} independent of $R$ and $\rho$ (or independent of an upper bound on $R$ in the drift case), we need to investigate the scaling properties of the extremal operators carefully, which will be done in the following auxiliary results.  Furthermore, since we need to prove the result at a generic scale, $R$, the proof deviates from that in \cite{Sil14} by a significant amount, and that is why we include the full proof of Theorems \ref{thmHarnack:WeakHarnackDrift} and \ref{thmHarnack:WeakHarnackNoDrift} below.

\begin{lemma}\label{lemmaHarnack:BumpFunctionOperatorBound}
Let $C,R>0$ and $\varphi\in C_b^2(\Rd)$ satisfy 
\begin{enumerate}[label=(\roman*)]
\item $\|\varphi\|_{L^\infty(\Rd)}\leq C$,
\item $\|\nabla\varphi\|_{L^\infty(\Rd)}\leq CR^{-1}$ and
\item $\|D^2\varphi\|_{L^\infty(\Rd)}\leq CR^{-2}$.
\end{enumerate}
Then $\M_\rho^\pm\varphi\in L^\infty(\Rd)$ with
\begin{align}\label{eqHarnack:BumpFunctionOperator}
\|\M_\rho^\pm\varphi\|_{L^\infty(\Rd)}\leq\tilde{C}R^{-2s}
\end{align}
for some $\tilde{C}=\tilde{C}(s,\Lambda,C)>0$.
\end{lemma}

The proof, which can be found in Appendix \ref{secApp:Proofs}, does not use the lower bound \eqref{eqPre:LowerEllipticityAssumption} such that it does not differ from the known case, where $\rho=0$.

\begin{proposition}\label{propHarnack:HölderPucci}
Fix $\alpha\in (0,1]$. Then there exists a constant $C=C(s,\Lambda,\alpha)>0$ such that the following statements hold true.
\begin{enumerate}[label=(\alph*)]
\item If $s<1/2$ and $\alpha\in(2s,1]$, then, for any $\varphi\in C^{0,\alpha}(\Rd)$, we have $\M_\rho^\pm\varphi\in C^{0,\alpha-2s}(\Rd)$ and furthermore,
\begin{align*}
[\M_\rho^\pm\varphi]_{C^{0,\alpha-2s}(\Rd)}\leq C[\varphi]_{C^{0,\alpha}(\Rd)}.
\end{align*}\label{itemHarnack:HölderPucciSmallS}
\item If $s>1/2$ and $\alpha\in(2s-1,1]$, then, for any $\varphi\in C^{1,\alpha}(\Rd)$, we have $\M_\rho^\pm\varphi\in C^{0,1+\alpha-2s}(\Rd)$ and furthermore,
\begin{align*}
[\M_\rho^\pm\varphi]_{C^{0,1+\alpha-2s}(\Rd)}\leq C[\varphi]_{C^{1,\alpha}(\Rd)}.
\end{align*}\label{itemHarnack:HölderPucciLargeS}
\item If $s=1/2$ and $\alpha\in(0,1)$, then, for any $\varphi\in C^{1,\alpha}(\Rd)$, we have $\M_\rho^\pm\varphi\in C^{0,\alpha}(\Rd)$ and furthermore,
\begin{align*}
[\M_\rho^\pm\varphi]_{C^{0,\alpha}(\Rd)}\leq C[\varphi]_{C^{1,\alpha}(\Rd)}.
\end{align*}\label{itemHarnack:HölderPucciCriticalS}
\end{enumerate}
\end{proposition}

For our purposes, it would be sufficient to use the results for $\alpha=1$. However, for $s=1/2$, we need to take a different $\alpha$ anyway. Therefore, we state the above proposition in a wider generality, also because we believe it might be useful in other contexts. We shift the proof to Appendix \ref{secApp:Proofs}, as well.

The above estimates use the corresponding seminorms in all of $\Rd$, which is sufficient for us, since we will apply them for compactly supported functions, below. However, one can derive similar results for localized estimates by following the approach of \cite[Lemma 1.2.3]{FR24}.

It is interesting that the choice $\alpha=1$ in Proposition \ref{propHarnack:HölderPucci} is impossible in the case $s=1/2$.

\begin{example}
Consider the case of the fractional $1/2$-Laplacian on a line, defined by
\begin{align*}
(-\Delta)^{1/2}u(x)\coloneqq\frac{1}{2}\int_{-\infty}^\infty\frac{2u(x)-u(x+y)-u(x-y)}{y^2}\d y, \quad x\in\R,
\end{align*}
for sufficiently regular functions $u$. More precisely, we take $d=1$, $\rho=0$ and $\lambda=1=\Lambda$ so that $\M_0^+=\M_0^-=-(-\Delta)^{1/2}$. Now, define
\begin{align*}
u(x)\coloneqq\frac{1}{2}x|x|\eta(x), \quad x\in\R,
\end{align*}
where $\eta\in C_c^\infty(\R)$ is a bump function with $0\leq\eta\leq 1$ and $\eta\equiv 1$ on $[-1,1]$. Then it is easy to see that $u\in C^{1,1}(\R)$ with
\begin{align*}
u^\prime(x)=|x|\eta(x)+\frac{1}{2}x|x|\eta^\prime(x), \quad x\in\R.
\end{align*}
Moreover, the $1/2$-Laplacian in one dimension can be computed as
\begin{align*}
(-\Delta)^{1/2}u(x)&=\int_0^\infty\frac{(u(x)-u(x+y))+(u(x)-u(x-y))}{y^2}\d y \\
&=-\int_0^\infty\int_0^y\frac{u^\prime(x+z)-u^\prime(x-z)}{y^2}\d z\d y=-\int_0^\infty\int_z^\infty\frac{u^\prime(x+z)-u^\prime(x-z)}{y^2}\d y\d z \\
&=-\int_0^\infty\frac{u^\prime(x+z)-u^\prime(x-z)}{z}\d z=\int_{-\infty}^\infty\frac{u^\prime(y)}{x-y}\d y.
\end{align*}
Let $x\in(-1,1)$. Since
\begin{align*}
\int_{-1}^1\frac{u^\prime(y)}{x-y}\d y&=\int_{-1}^1\frac{|y|}{x-y}\d y=2x\log|x|-x\log(1-x^2),
\end{align*}
we see that
\begin{align*}
(-\Delta)^{1/2}u(x)=2x\log|x|-x\log(1-x^2)+\int_{\R\setminus(-1,1)}\frac{|y|\eta(y)+\frac{1}{2}y|y|\eta^\prime(y)}{x-y}\d y
\end{align*}
with the convention $0\log0=0$. This function is indeed no longer Lipschitz continuous around $x=0$. In fact, note that $x\log(1-x^2)$ as well as the integral term are smooth on $[-1/2,1/2]$. The remaining term $2x\log|x|$ is, however, not $C^{0,1}$ in any neighborhood of the origin, see \cite[p.\ 187]{FR22}.
\end{example}

It will be convenient to use Proposition \ref{propHarnack:HölderPucci} in the form of the following corollary. In fact, if $\alpha=1$ were an admissible choice in Proposition \ref{propHarnack:HölderPucci} for all $s$, Corollary \ref{corHarnack:HölderPucci} would be immediate. We note that this is needed because we prove our results at a generic scale without invoking a rescaling.  That is why this type of result is not typically a part of other works establishing similar results.  Below it will be used to propagate advantageous values of $\M_\rho^\pm$ across a small distance.

\begin{corollary} \label{corHarnack:HölderPucci}
Let $C,R>0$ and $\varphi\in C_c^2(\Rd)$ satisfy
\begin{enumerate}[label=(\roman*)]
\item $\supp\varphi\subseteq B_R$,
\item $\|\nabla\varphi\|_{L^\infty(\Rd)}\leq CR^{-1}$ and
\item $\|D^2\varphi\|_{L^\infty(\Rd)}\leq CR^{-2}$.
\end{enumerate}
Then there exist $\tilde{C}=\tilde{C}(s,\Lambda,C)>0$ and $\gamma=\gamma(s)>0$ such that
\begin{align*}
|\M_\rho^\pm\varphi(x)-\M_\rho^\pm\varphi(y)|\leq\frac{\tilde{C}}{R^{2s}}\left(\frac{|x-y|}{R}\right)^\gamma.
\end{align*}
\end{corollary}
\begin{proof}
Assume first that $s<1/2$. By Proposition \ref{propHarnack:HölderPucci}\ref{itemHarnack:HölderPucciSmallS} for $\alpha=1$, there exists $C_1=C_1(s,\Lambda)>0$ such that, for all $x,y\in\Rd$,
\begin{align*}
|\M_\rho^\pm\varphi(x)-\M_\rho^\pm\varphi(y)|&\leq C_1[\varphi]_{C^{0,1}(\Rd)}|x-y|^{1-2s}=C_1\|\nabla\varphi\|_{L^\infty(\Rd)}|x-y|^{1-2s} \\
&\leq\frac{C_1C}{R}|x-y|^{1-2s}=\frac{C_1C}{R^{2s}}\left(\frac{|x-y|}{R}\right)^{1-2s},
\end{align*}
which proves that we can choose $\gamma=1-2s$ here.

Suppose now that $s>1/2$. Then we apply Proposition \ref{propHarnack:HölderPucci}\ref{itemHarnack:HölderPucciLargeS} for $\alpha=1$ and obtain
\begin{align*}
|\M_\rho^\pm\varphi(x)-\M_\rho^\pm\varphi(y)|&\leq C_1[\varphi]_{C^{1,1}(\Rd)}|x-y|^{2-2s}=C_1\|D^2\varphi\|_{L^\infty(\Rd)}|x-y|^{2-2s} \\
&\leq\frac{C_1C}{R^2}|x-y|^{2-2s}=\frac{C_1C}{R^{2s}}\left(\frac{|x-y|}{R}\right)^{2-2s}
\end{align*}
for any $x,y\in\Rd$, that is, we take $\gamma=2-2s$.

Finally, for $s=1/2$, we simply take $\alpha=1/2$ and use Proposition \ref{propHarnack:HölderPucci}\ref{itemHarnack:HölderPucciCriticalS} to get
\begin{align*}
|\M_\rho^\pm\varphi(x)-\M_\rho^\pm\varphi(y)|\leq C_2[\varphi]_{C^{1,1/2}(\Rd)}|x-y|^{1/2}
\end{align*}
for some $C_2=C_2(\Lambda)>0$. The seminorm on the right hand side can be computed in the following way. In $B_{2R}$, we have
\begin{align*}
\sup_{\substack{x,y\in B_{2R}\\x\neq y}}\frac{|\nabla\varphi(x)-\nabla\varphi(y)|}{|x-y|^{1/2}}\leq\sup_{\substack{x,y\in\Rd\\ 0<|x-y|<4R}}\frac{\|D^2\varphi\|_{L^\infty(\Rd)}|x-y|}{|x-y|^{1/2}}\leq\frac{C}{R^2}(4R)^{1/2}=\frac{2C}{R^{3/2}}.
\end{align*}
On the other hand, if at least one of the two points does not belong to $B_{2R}$, in particular, does not belong to the support of $\varphi$, we get
\begin{align*}
\sup_{\substack{x,y\in\Rd\\x\neq y\\ x\notin B_{2R}\ \text{or}\ y\notin B_{2R}}}\frac{|\nabla\varphi(x)-\nabla\varphi(y)|}{|x-y|^{1/2}}=\sup_{\substack{x,y\in\Rd\\x\in B_R\\y\notin B_{2R}}}\frac{|\nabla\varphi(x)|}{|x-y|^{1/2}}\leq\|\nabla\varphi\|_{L^\infty(\Rd)}\sup_{\substack{x,y\in\Rd\\|x-y|\geq R}}|x-y|^{-1/2}\leq\frac{C}{R^{3/2}}
\end{align*}
because $\nabla\varphi$ vanishes outside of $B_R$. Hence, we found that
\begin{align*}
|\M_\rho^\pm\varphi(x)-\M_\rho^\pm\varphi(y)|\leq\frac{2CC_2}{R^{3/2}}|x-y|^{1/2}=\frac{2CC_2}{R}\left(\frac{|x-y|}{R}\right)^{1/2}
\end{align*}
for all $x,y\in\Rd$. In other words, the claim is proven for $\gamma=1/2$ in this case.

Taking $\tilde{C}=\max\{C_1C,2CC_2\}=\tilde{C}(s,\Lambda,C)>0$ finishes the proof.
\end{proof}

We have now collected all necessary facts to prove the weak Harnack inequality.  This proof starts out with the same structure from \cite{Sil14}, and then goes into additional technical arguments due to the generic scale.

\begin{proof}[Proof of Theorem \ref{thmHarnack:WeakHarnackDrift}]
If $u\equiv 0$ in $\Rd\times(I_R\setminus I_{R/2})$, then the claim is fulfilled trivially. Assume that $u>0$ somewhere. Let us first claim that it suffices to prove \eqref{eqHarnack:Result} for functions $u$ that satisfy
\begin{align}\label{eqHarnack:WlogAssumption}
\int_{I_R\setminus I_{R/2}}\intd\frac{u(x,t)}{(R+|x|)^{d+2s}}\d x\d t=1.
\end{align}
By lower semicontinuity, there exists a set of positive measure in $\Rd\times(I_R\setminus I_{R/2})$, where $u>0$. Now we can consider the function
\begin{align*}
\tilde{u}\coloneqq\frac{u}{\int_{I_R\setminus I_{R/2}}\intd\frac{u(x,t)}{(R+|x|)^{d+2s}}\d x\d t}.
\end{align*}
This function is also globally bounded and lower semicontinuous in $Q_R$. Moreover, since the equation is positively homogeneous, $\tilde{u}$ solves
\begin{align*}
\partial_t\tilde{u}+\Lambda|\nabla\tilde{u}|-\M_\rho^-\tilde{u}\geq -R^{-2s}\frac{a}{\int_{I_R\setminus I_{R/2}}\intd\frac{u(x,t)}{(R+|x|)^{d+2s}}\d x\d t}\enspace\text{in}\enspace Q_R
\end{align*}
in the viscosity sense. As $\tilde{u}$ satisfies \eqref{eqHarnack:WlogAssumption}, we could conclude
\begin{align*}
\inf_{Q_{R/2}}\tilde{u}\geq c-\frac{a}{\int_{I_R\setminus I_{R/2}}\intd\frac{u(x,t)}{(R+|x|)^{d+2s}}\d x\d t},
\end{align*}
which is equivalent to \eqref{eqHarnack:Result}. Hence, we will now assume without loss of generality that \eqref{eqHarnack:WlogAssumption} is true. The assertion reduces to
\begin{align*}
\inf_{Q_{R/2}}u\geq c-a
\end{align*}
for some universal constant $c>0$. Choose $\theta=\theta(d,s)>0$ sufficiently small such that
\begin{align}\label{eqHarnack:ChoiceTheta}
\int_{I_R\setminus I_{R/2}}\intd\frac{\theta}{(R+|x|)^{d+2s}}\d x\d t\leq\frac{1}{2}.
\end{align}
Indeed, this is choice can be taken independently of $R$ because
\begin{align*}
\int_{I_R\setminus I_{R/2}}\intd\frac{\d x\d t}{(R+|x|)^{d+2s}}\leq R^{2s}\intd\frac{\d x}{(R+\abs{x})^{d+2s}}\leq\omega_d R^{2s}\int_R^\infty r^{-1-2s}\d r=\frac{\omega_d}{2s},
\end{align*}
where $\omega_d$ denotes the surface measure of the unit sphere in $\Rd$. Thus, \eqref{eqHarnack:WlogAssumption} and \eqref{eqHarnack:ChoiceTheta} imply
\begin{align}\label{eqHarnack:ChoiceThetaConsequence}
\int_{I_R\setminus I_{R/2}}\intd\frac{(u(x,t)-\theta)_+}{(R+|x|)^{d+2s}}\d x\d t\geq\frac{1}{2}.
\end{align}
Fix some constants $c_0,C_1>0$ and let $m\colon\bar{I}_R\to\R$ be the solution of the following ODE
\begin{align}\label{eqHarnack:ODE}
\begin{dcases}
	m'(t)=c_0\intd\frac{(u(x,t)-\theta)_+}{(R+|x|)^{d+2s}}\d x-C_1R^{-2s}m(t),\ &t\in I_R\setminus I_{R/2}, \\
	m'(t)=-C_1R^{-2s}m(t), \ &t\in I_{R/2}, \\
	m(-R^{2s})=0 
\end{dcases}
\end{align}
The function $m$ accumulates the spatial mass of $u$ over time to build a lower bound. It is easy to check that the solution is explicitly given by
\begin{align*}
m(t)=\begin{dcases}
	c_0\int_{-R^{2s}}^t\intd\frac{(u(x,\tau)-\theta)_+}{(R+|x|)^{d+2s}}\exp\left(-\frac{C_1}{R^{2s}}(t-\tau)\right)\d x\d\tau,\ &t\in\bar{I}_R\setminus I_{R/2}, \\
	m\left(-(R/2)^{2s}\right)\exp\left(-\frac{C_1}{R^{2s}}\left(t+(R/2)^{2s}\right)\right),\ &t\in I_{R/2}.
\end{dcases}
\end{align*}
Note that \eqref{eqHarnack:WlogAssumption} gives that, for $t\in\overline{I}_R\setminus I_{R/2}$,
\begin{align*}
m(t)\leq c_0\int_{I_R\setminus I_{R/2}}\intd\frac{u(x,\tau)}{(R+|x|)^{d+2s}}\d x\d\tau=c_0,
\end{align*}
whence
\begin{align*}
m(t)=m\left(-(R/2)^{2s}\right)\exp\left(-\frac{C_1}{R^{2s}}\left(t+(R/2)^{2s}\right)\right)\leq c_0
\end{align*}
for $t\in I_{R/2}$ as well. We will choose $c_0\leq\theta$ to achieve that $m\leq\theta$ in $\overline{I}_R$. We can also find a lower bound for $m$ on $I_{R/2}$. Indeed, by \eqref{eqHarnack:ChoiceThetaConsequence}, we find
\begin{align*}
m\left(-(R/2)^{2s}\right)&=c_0\int_{-R^{2s}}^{-(R/2)^{2s}}\intd\frac{u(x,\tau)-\theta)_+}{(R+|x|)^{d+2s}}\exp\left(-\frac{C_1}{R^{2s}}\left(-(R/2)^{2s}-\tau\right)\right)\d x\d\tau \\
&\geq c_0\int_{I_R\setminus I_{R/2}}\intd\frac{u(x,\tau)-\theta)_+}{(R+|x|)^{d+2s}}\d x\d\tau\exp\left(-\frac{C_1}{R^{2s}}\left(-(R/2)^{2s}+R^{2s}\right)\right) \\
&\geq\frac{c_0}{2}\exp\left(-\frac{C_1}{R^{2s}}\left(R^{2s}-(R/2)^{2s}\right)\right)
\end{align*}
such that
\begin{align}
m(t)&\geq\frac{c_0}{2}\exp\left(-\frac{C_1}{R^{2s}}\left(R^{2s}-(R/2)^{2s}\right)\right)\exp\left(-\frac{C_1}{R^{2s}}\left(t+(R/2)^{2s}\right)\right) \nonumber \\
&=\frac{c_0}{2}\exp\left(-\frac{C_1}{R^{2s}}\left(R^{2s}+t\right)\right) \nonumber \\
&\geq\frac{c_0}{2}e^{-C_1}\label{eqHarnack:LowerBoundm}
\end{align}
for all $t\in I_{R/2}$.

Let $\beta\in C^\infty([0,\infty))$ be a function satisfying $0\leq\beta\leq 1$, $\beta\equiv 1$ on $[0,1/2]$, $\supp\beta=[0,1]$ and $-2\leq\beta'\leq 0$. We use $\beta$ to construct a bump function
\begin{align*}
\eta(x)\coloneqq\beta\left(\frac{|x|}{R}\right), \quad x\in\Rd.
\end{align*}
Note that $\eta\in C_c^\infty(\Rd)$ satisfies $0\leq\eta\leq 1$, $\eta\equiv 1$ in $\overline{B}_{R/2}$, $\supp\eta=\overline{B}_R$, $\|\nabla\eta\|_{L^\infty}\leq CR^{-1}$ and $\|D^2\eta\|_{L^\infty}\leq CR^{-2}$, where $C=C(d)>0$ is some constant. It now suffices to prove that
\begin{align*}
u(x,t)\geq m(t)\eta(x)-R^{-2s}(R^{2s}+t)a \quad\forall (x,t)\in Q_{R/2}
\end{align*}
because then it follows from \eqref{eqHarnack:LowerBoundm} that
\begin{align*}
\inf_{Q_{R/2}}u\geq\left(\inf_{I_{R/2}}m\right)\left(\inf_{B_{R/2}}\eta\right)-R^{-2s}\sup_{t\in I_{R/2}}\left(R^{2s}+t\right)a\geq\frac{c_0}{2}e^{-C_1}-a,
\end{align*}
which would finish the proof for $c=c_0e^{-C_1}/2$. Now, assume for contradiction that this assertion is false. Then, for sufficiently small $\varepsilon>0$, we get $u<\varphi$ at some point in $Q_{R/2}$, where $\varphi$ is defined as
\begin{align}\label{eqHarnack:DefinitionTest}
\varphi(x,t)\coloneqq m(t)\eta(x)-R^{-2s}\left(R^{2s}+t\right)a-\varepsilon
\end{align}
for $(x,t)\in\Rd\times\overline{I}_R$. Since $m(-R^{2s})=0$ and $\eta$ is supported in $\overline{B}_R$, we have $\varphi<0$ outside $Q_R$. By the assumption that $u\geq 0$ in $\Rd\times\overline{I}_R$, it is true that $u>\varphi$ outside $Q_R$. On the other hand, we must have $u<\varphi$ somewhere in $Q_{R/2}$. By semicontinuity and eventually adjusting $\varepsilon$, we can find a first crossing point $(x_0,t_0)\in Q_R$ such that, for $\varphi$ as in \eqref{eqHarnack:DefinitionTest},
\begin{align}\label{eqHarnack:FirstCrossingPoint}
u(x_0,t_0)=\varphi(x_0,t_0) \quad\text{and}\quad u(x,t)\geq\varphi(x,t) \enspace\text{for all}\enspace (x,t)\in\Rd\times(-R^{2s},t_0].
\end{align}
For complete presentation, we include the justification of (\ref{eqHarnack:FirstCrossingPoint}) in the appendix as Lemma \ref{lemmaAppendix:JustificationFirstCrossingPoint}. We have found that the function $\varphi$ is in fact a test function for \eqref{eqHarnack:EquationDrift} at $(x_0,t_0)$ and moreover, we obtain
\begin{multline}\label{eqHarnack:EquationTestFunction}
\partial_t\varphi(x_0,t_0)+\Lambda|\nabla\varphi(x_0,t_0)|-\M_\rho^-\varphi(x_0,t_0) \\
-\intd(u(x_0+y,t_0)-\varphi(x_0+y,t_0))\min\left\{\frac{\lambda}{|y|^{d+2s}},\frac{\lambda}{\rho^{d+2s}}\right\}\d y\geq -R^{-2s}a
\end{multline}
by Lemma \ref{lemmaPre:GlobalTestFunction}. Let us compute the appearing quantities. First, by definition of $\varphi$, we have
\begin{align*}
\partial_t\varphi(x_0,t_0)=m'(t_0)\eta(x_0)-R^{-2s}a\enspace\text{and}\enspace\nabla\varphi(x_0,t_0)=m(t_0)\nabla\eta(x_0)
\end{align*}
as well as
\begin{align*}
\M_\rho^-\varphi(x_0,t_0)&=\inf_{K\in\K_\rho}\intd\delta_y\varphi(x_0,t_0)K(y)\d y \\
&=\inf_{K\in\K_\rho}\intd m(t_0)\delta_y\eta(x_0)K(y)\d y=m(t_0)\M_\rho^-\eta(x_0).
\end{align*}
By construction and a change of variables, we can also obtain that
\begin{align*}
&\intd(u(x_0+y,t_0)-\varphi(x_0+y,t_0))\min\left\{\frac{\lambda}{|y|^{d+2s}},\frac{\lambda}{\rho^{d+2s}}\right\}\d y \\
&\quad =\intd(u(x,t_0)-\varphi(x,t_0))_+\min\left\{\frac{\lambda}{|x-x_0|^{d+2s}},\frac{\lambda}{\rho^{d+2s}}\right\}\d x \\
&\quad =\intd(u(x,t_0)-m(t_0)\eta(x)+R^{-2s}(R^{2s}+t_0)a+\varepsilon)_+\min\left\{\frac{\lambda}{|x-x_0|^{d+2s}},\frac{\lambda}{\rho^{d+2s}}\right\}\d x.
\end{align*}
Recall now that $m\leq\theta$, $\eta\leq 1$ and $t_0\geq -R^{2s}$. Hence, the fact that $R>\rho$ yields
\begin{align*}
&\intd(u(x_0+y,t_0)-\varphi(x_0+y,t_0))\min\left\{\frac{\lambda}{|y|^{d+2s}},\frac{\lambda}{\rho^{d+2s}}\right\}\d y \\
&\quad\geq\lambda\intd(u(x,t_0)-\theta)_+\inf_{y_0\in B_R}\min\left\{|x-y_0|^{-d-2s},\rho^{-d-2s}\right\}\d x \\
&\quad =\lambda\intd(u(x,t_0)-\theta)_+\min\left\{(R+|x|)^{-d-2s},\rho^{-d-2s}\right\}\d x \\
&\quad =\lambda\intd\frac{(u(x,t_0)-\theta)_+}{(R+|x|)^{d+2s}}\d x.
\end{align*}
Here we used that the infimum is attained for $y_0=-Rx/|x|\in\partial B_R$ for each $x\in\Rd\setminus\{0\}$. Substituting those computations back to \eqref{eqHarnack:EquationTestFunction}, we get
\begin{align}\label{eqHarnack:InequalityToContradict}
m'(t_0)\eta(x_0)+\Lambda m(t_0)|\nabla\eta(x_0)|-m(t_0)\M_\rho^-\eta(x_0)-\lambda\intd\frac{(u(x,t_0)-\theta)_+}{(R+|x|)^{d+2s}}\d x\geq 0.
\end{align}
In order to get a contradiction, we need to distinguish between the cases $t_0\in I_R\setminus I_{R/2}$ and $t_0\in I_{R/2}$. Assume first that $t_0\in I_R\setminus I_{R/2}$. By \eqref{eqHarnack:ODE}, \eqref{eqHarnack:InequalityToContradict} becomes
\begin{multline*}
\left(c_0\intd\frac{(u(x,t_0)-\theta)_+}{(R+|x|)^{d+2s}}\d x-C_1R^{-2s}m(t_0)\right)\eta(x_0) \\
+\Lambda m(t_0)|\nabla\eta(x_0)|-m(t_0)\M_\rho^-\eta(x_0)-\lambda\intd\frac{(u(x,t_0)-\theta)_+}{(R+|x|)^{d+2s}}\d x\geq 0,
\end{multline*}
or equivalently,
\begin{align*}
(c_0\eta(x_0)-\lambda)\intd\frac{(u(x,t_0)-\theta)_+}{(R+|x|)^{d+2s}}\d x-m(t_0)\left(C_1R^{-2s}\eta(x_0)-\Lambda|\nabla\eta(x_0)|+\M_\rho^-\eta(x_0)\right)\geq 0.
\end{align*}
Let us also choose $c_0\leq\lambda$, that is, $c_0\coloneqq\min\{\theta,\lambda\}=c_0(d,s,\lambda)>0$. By estimating $\eta(x_0)\leq 1$, we can get rid of the first summand. Moreover, it is clear that $m(t_0)>0$, otherwise this would contradict \eqref{eqHarnack:FirstCrossingPoint}. Hence, we arrive at
\begin{align}\label{eqHarnack:FinalInequality}
-C_1R^{-2s}\eta(x_0)\geq \M_\rho^-\eta(x_0)-\Lambda|\nabla\eta(x_0)|.
\end{align}
On the other hand, if $t_0\in I_{R/2}$, then \eqref{eqHarnack:ODE} and \eqref{eqHarnack:InequalityToContradict} yield
\begin{align*}
-C_1R^{-2s}m(t_0)\eta(x_0)+\Lambda m(t_0)|\nabla\eta(x_0)|-m(t_0)\M_\rho^-\eta(x_0)-\lambda\intd\frac{(u(x,t_0)-\theta)_+}{(R+|x|)^{d+2s}}\d x\geq 0.
\end{align*}
Since the integral is nonnegative, we can neglect it and divide again by $m(t_0)>0$ to obtain \eqref{eqHarnack:FinalInequality} also in this case. We need to choose $C_1>0$ large enough to find a contradiction to \eqref{eqHarnack:FinalInequality}. Let $C_2=C_2(d,s,\Lambda)>0$ be the constant from Lemma \ref{lemmaHarnack:BumpFunctionOperatorBound}. Fix some $\sigma\in(0,1)$ and assume that $x_0\in B_{\sigma R}$. By construction of $\eta$ and $\beta$, we know that then $\eta(x_0)\geq\eta_0$, where
\begin{align}\label{eqHarnack:DefEta0}
\eta_0\coloneqq\inf_{B_{\sigma R}}\eta=\inf_{x\in B_{\sigma R}}\beta\left(\frac{|x|}{R}\right)=\beta\left(\frac{\sigma R}{R}\right)=\beta(\sigma)=\eta_0(\sigma)>0.
\end{align}
Notice that we will choose $\sigma$ at the very end of the proof. Since $s\geq 1/2$, we have
\begin{align*}
\|\nabla\eta\|_{L^\infty(B_R)}\leq CR^{-1}=C\left(\frac{R}{R_0}\right)^{-1}R_0^{-1}\leq C\left(\frac{R}{R_0}\right)^{-2s}R_0^{-1}=CR_0^{2s-1}R^{-2s}
\end{align*}
as $R\leq R_0$. Hence,
\begin{align*}
-C_1R^{-2s}\eta(x_0)&\leq -C_1R^{-2s}\eta_0< -\left(C_2+C\Lambda R_0^{2s-1}\right)R^{-2s} \\
&\leq -\|\M_\rho^-\eta\|_{L^\infty(\Rd)}-\Lambda\|\nabla\eta\|_{L^\infty(B_R)}\leq\M_\rho^-\eta(x_0)-\Lambda|\nabla\eta(x_0)|,
\end{align*}
provided we have chosen
\begin{align*}
C_1=1+\frac{C_2+C\Lambda R_0^{2s-1}}{\eta_0}=C_1(d,s,\Lambda,R_0,\sigma)>0.
\end{align*}
This is a contradiction to \eqref{eqHarnack:FinalInequality}.

We are now left to deal with the case $x_0\in B_R\setminus B_{\sigma R}$. We will deduce a lower bound for $\M_\rho^-\eta$ on $\partial B_R$ and apply Corollary \ref{corHarnack:HölderPucci}. For that, denote
\begin{align}\label{eqHarnack:DefinitionMu}
\mu\coloneqq 2^{-d-2s}\lambda\int_{B_1}\beta(|y|)\d y=\mu(d,s,\lambda)>0.
\end{align}
By \eqref{eqPre:LowerEllipticityAssumption} and the fact that $R>\rho$, we get that
\begin{align}
\min_{\partial B_R}\M_\rho^-\eta&=\min_{x\in\partial B_R}\inf_{K\in\K_\rho}\intd\eta(x+y)K(y)\d y=\min_{x\in\partial B_R}\inf_{K\in\K_\rho}\intd\beta\left(\frac{|x+y|}{R}\right)K(y)\d y \nonumber \\
&\geq\inf_{K\in\K_\rho}\int_{B_R}\inf_{x\in\partial B_R}\beta\left(\frac{|y|}{R}\right)K(y-x)\d y \nonumber \\
&\geq\int_{B_R}\beta\left(\frac{|y|}{R}\right)\min_{x\in\partial B_R}\lambda\min\left\{|y-x|^{-d-2s},\rho^{-d-2s}\right\}d y \nonumber \\
&\geq\int_{B_R}\beta\left(\frac{|y|}{R}\right)\lambda\min\left\{(2R)^{-d-2s},\rho^{-d-2s}\right\}d y \nonumber \\
&=\int_{B_R}\beta\left(\frac{|y|}{R}\right)\frac{\lambda}{(2R)^{d+2s}}d y=\int_{B_1}\beta(|y|)R^{-2s}\frac{\lambda}{2^{d+2s}}\d y=\mu R^{-2s}.\label{eqHarnack:LowerBoundPucci}
\end{align}
Let us now apply Corollary \ref{corHarnack:HölderPucci} to $\eta$, that is, there exist $C_3=C_3(d,s,\Lambda)>0$ and $\gamma=\gamma(s)>0$ such that
\begin{align}\label{eqHarnack:PucciHölder}
|\M_\rho^-\eta(x)-\M_\rho^-\eta(y)|\leq\frac{C_3}{R^{2s}}\left(\frac{|x-y|}{R}\right)^{\gamma}
\end{align}
for all $x,y\in\Rd$. Since $y\coloneqq x_0\neq 0$, we can choose $x\coloneqq Rx_0/|x_0|\in\partial B_R$ because then
\begin{align*}
\mu R^{-2s}-\M_\rho^-\eta(x_0)\leq|\M_\rho^-\eta(x)-\M_\rho^-\eta(x_0)|\leq\frac{C_3}{R^{2s}}\left(\frac{|x-x_0|}{R}\right)^\gamma=\frac{C_3}{R^{2s}}\left(\frac{R-|x_0|}{R}\right)^\gamma
\end{align*}
by \eqref{eqHarnack:LowerBoundPucci} and \eqref{eqHarnack:PucciHölder}. Choosing $\sigma\geq 1-(\mu/(2C_3))^{1/\gamma}$ yields
\begin{align}
\M_\rho^-\eta(x_0)&\geq\mu R^{-2s}-\frac{C_3}{R^{2s}}\left(\frac{R-|x_0|}{R}\right)^\gamma\geq\mu R^{-2s}-\frac{C_3}{R^{2s}}\left(\frac{R-\sigma R}{R}\right)^\gamma \nonumber \\
&=\mu R^{-2s}-\frac{C_3}{R^{2s}}(1-\sigma)^\gamma\geq\mu R^{-2s}-\frac{\mu}{2}R^{-2s}=\frac{\mu}{2}R^{-2s}.\label{eqHarnack:FinalEstimatePucci}
\end{align}
We can proceed with $|\nabla\eta|$ in the same way. Here, $|\nabla\eta(x)|=0$ as $x\in\partial B_R$ such that, by construction,
\begin{align*}
|\nabla\eta(x_0)|=\big||\nabla\eta(x)|-|\nabla\eta(x_0)|\big|\leq\|D^2\eta\|_{L^\infty(\Rd)}|x-x_0|\leq\frac{C}{R^2}(R-|x_0|).
\end{align*}
Now, we will also impose $\sigma\geq 1-\mu/(4C\Lambda R_0^{2s-1})$ to obtain
\begin{align}\label{eqHarnack:FinalEstimateDrift}
\Lambda|\nabla\eta(x_0)|\leq\Lambda\frac{C}{R^2}(R-\sigma R)\leq\frac{\mu}{4R_0^{2s-1}}R^{-1}=\frac{\mu}{4R_0^{2s}}\left(\frac{R}{R_0}\right)^{-1}\leq\frac{\mu}{4R_0^{2s}}\left(\frac{R}{R_0}\right)^{-2s}=\frac{\mu}{4}R^{-2s}
\end{align}
due to $R\leq R_0$ and $s\geq 1/2$. Note that $\sigma=\sigma(d,s,\lambda,\Lambda,R_0)\in(0,1)$ and we found that, by \eqref{eqHarnack:FinalEstimatePucci} and \eqref{eqHarnack:FinalEstimateDrift},
\begin{align*}
\M_\rho^-\eta(x_0)-\Lambda|\nabla\eta(x_0)|\geq\frac{\mu}{2}R^{-2s}-\frac{\mu}{4}R^{-2s}>0\geq -C_1R^{-2s}\eta(x_0),
\end{align*}
which contradicts \eqref{eqHarnack:FinalInequality} such that the proof is completed. Note finally that $c_0$ and $C_1$ depend only on $d,s,\lambda,\Lambda$ and $R_0$ such that $c$ does so.
\end{proof}

As one can see, the proof above is more involved than the one in \cite[Theorem 5.1]{Sil14}. The main difficulty arises in choosing the parameter $\eta_0$ appropriately. On unit scale, it was sufficient to note that $\M_\rho^-\eta$ is positive at a global minimum such that by continuity, $\eta_0$ could be chosen sufficiently small to ensure that $\M_\rho^-\eta(x_0)$ is still positive in the case $\eta(x_0)<\eta_0$. We emphasize that in \cite{Sil14}, $\eta_0$ then only depended on the choice of the bump function, whereas in our case, it would also depend on $R$ and $\rho$. As a consequence, we needed  to study mapping properties of $\M_\rho^\pm$ that are robust as $\rho\to 0^+$. An additional necessity was to choose the bump function to gain lower bounds on $\M_\rho^-\eta$ independently of $\rho$ as in \eqref{eqHarnack:LowerBoundPucci}.

If the drift term is not present, the proof simplifies and additionally, we are able to deduce weak Harnack inequality for $s<1/2$.

\begin{proof}[Proof of Theorem \ref{thmHarnack:WeakHarnackNoDrift}]
We will perform the same proof as for Theorem \ref{thmHarnack:WeakHarnackDrift} except for neglecting the term $\Lambda|\nabla u|$. There will be no differences in the first half of the proof. One can easily check that we can get to \eqref{eqHarnack:FinalInequality}, which reads
\begin{align}\label{eqHarnack:InequalityToContradictNoDrift}
-C_1R^{-2s}\eta(x_0)\geq\M_\rho^-\eta(x_0)
\end{align}
in this scenario. Recall that we have to choose $C_1=C_1(d,s,\lambda,\Lambda)>0$ to obtain a contradiction to \eqref{eqHarnack:InequalityToContradictNoDrift}. Define, as before, $\eta_0$ by \eqref{eqHarnack:DefEta0}, where $\sigma\in(0,1)$ needs to be chosen below. If $\eta(x_0)\geq\eta_0$, then
\begin{align*}
-C_1R^{-2s}\eta(x_0)\leq -C_1R^{-2s}\eta_0<-C_2R^{-2s}\leq -\|\M_\rho^-\eta\|_{L^\infty(\Rd)}\leq\M_\rho^-\eta(x_0)
\end{align*}
if we choose $C_1>C_2/\eta_0$, which already provides a contradiction to \eqref{eqHarnack:InequalityToContradictNoDrift}. Here, $C_2=C_2(d,s,\Lambda)>0$ is the constant appearing in Lemma \ref{lemmaHarnack:BumpFunctionOperatorBound}. Let us consider $\eta(x_0)<\eta_0$ and choose $\sigma\in(0,1)$ sufficiently large to deduce a contradiction, as well. By Corollary \ref{corHarnack:HölderPucci}, we have
\begin{align}\label{eqHarnack:HölderEstimatePucci}
|\M_\rho^-\eta(x)-\M_\rho^-\eta(y)|\leq\frac{C_3}{R^{2s}}\left(\frac{|x-y|}{R}\right)^\gamma
\end{align}
for any $x,y\in\Rd$, some $C_3=C_3(s,\Lambda)>0$ and some $\gamma=\gamma(s)>0$. We apply this again for $y\coloneqq x_0$ and $x\coloneqq Rx_0/|x_0|\in\partial B_R$ such that \eqref{eqHarnack:LowerBoundPucci} and \eqref{eqHarnack:HölderEstimatePucci} yield
\begin{align*}
\mu R^{-2s}-\M_\rho^-\eta(x_0)\leq|\M_\rho^-\eta(x)-\M_\rho^-\eta(x_0)|\leq\frac{C_3}{R^{2s}}\left(\frac{|x-x_0|}{R}\right)^\gamma=\frac{C_3}{R^{2s}}\left(\frac{R-|x_0|}{R}\right)^\gamma
\end{align*}
with $\mu$ as in \eqref{eqHarnack:DefinitionMu}. Now we are in a good position to choose $\sigma$. If $\sigma\geq 1-(\mu/(2C_3))^{1/\gamma}$, we arrive at
\begin{align*}
\M_\rho^-\eta(x_0)&\geq\mu R^{-2s}-\frac{C_3}{R^{2s}}\left(\frac{R-|x_0|}{R}\right)^\gamma\geq\mu R^{-2s}-\frac{C_3}{R^{2s}}\left(\frac{R-\sigma R}{R}\right)^\gamma \\
&\geq\mu R^{-2s}-\frac{\mu}{2}R^{-2s}>0\geq -C_1R^{-2s}\eta(x_0),
\end{align*}
which contradicts \eqref{eqHarnack:InequalityToContradictNoDrift} as desired.
\end{proof}

\begin{remark}
The constant $c$ in Theorem \ref{thmHarnack:WeakHarnackDrift} and Theorem \ref{thmHarnack:WeakHarnackNoDrift} only depends on an upper bound, $R_0$, of the radius if a drift term is part of the equation. The reason comes from different scaling properties of the gradient and of the nonlocal operators. We have seen that, for $\varphi\in C_b^2(\Rd)$ as in Lemma \ref{lemmaHarnack:BumpFunctionOperatorBound},
\begin{align*}
\|\M_\rho^\pm\varphi\|_{L^\infty}\lesssim R^{-2s} \enspace\text{and}\enspace \|\nabla\varphi\|_{L^\infty}\lesssim R^{-1}.
\end{align*}
Recall that we want to obtain a weak Harnack inequality in small cylinders. However, if $s<1/2$, then $R^{-1}\ll R^{-2s}$ for small $R$, which explains why we cannot include a drift in this case because the drift dominates the equation in this sense. However, if $s>1/2$, then $R^{-2s}\ll R^{-1}$ for small radii. Therefore, as we have seen, we are able to control the contribution of the drift in order to deduce weak Harnack inequality. For large radii, we get $R^{-2s}\ll R^{-1}$ and the same issue as before occurs such that the constant $c$ in \eqref{eqHarnack:Result} would explode as $R\to\infty$. Note that it can be seen from the proof of Theorem \ref{thmHarnack:WeakHarnackDrift} that $c$ depends only on $R_0$ via the quantity $R_0^{2s-1}$, which means that, for $s=1/2$, the dependence on $R_0$ disappears. Indeed, this is clear, since the scaling order of the drift and the nonlocal operator are identical.
\end{remark}

We will use the weak Harnack inequality in the following form, which is a direct consequence as $u$ is globally nonnegative.

\begin{corollary}\label{corHarnack:HarnackCorollary}
Let $a,R_0>0$, $R\in(\rho,R_0]$ and $u\in L^\infty(\Rd\times I_R)$ with $u\geq 0$ in $\Rd\times\overline{I}_R$. Suppose that either
\begin{enumerate}[label=(\alph*)]
\item $s\geq 1/2$ and $u$ satisfies \eqref{eqHarnack:EquationDrift} in the viscosity sense or\label{itemHarnack:CorHarnackDrift}
\item $u$ satisfies \eqref{eqHarnack:EquationNoDrift} in the viscosity sense.\label{itemHarnack:CorHarnackNoDrift}
\end{enumerate}
Then
\begin{align*}
\inf_{Q_{R/2}}u\geq c\fint_{I_R\setminus I_{R/2}}\fint_{B_R}u(x,t)\d x\d t-a
\end{align*}
for some constant $c>0$ that depends on $d,s,\lambda$ and $\Lambda$ and in the case \ref{itemHarnack:CorHarnackDrift} also on $R_0$. Here we use the notation
\begin{align*}
\fint_Af(x)\d x\coloneqq\frac{1}{|A|}\int_Af(x)\d x.
\end{align*}
\end{corollary}


\section{Partial Hölder regularity}\label{sec:Hoelder}

Finally, we are able to deduce partial Hölder regularity of solutions. The partial nature of this estimate arises if $\rho>0$ because we do not have the weak Harnack inequality on arbitrary small scales, which means that we have to abort the upcoming oscillation iteration at some step. The result will be the following.

\begin{theorem}\label{thmHoelder:PartialHoelderAt0}
Let $A\geq 0$, $R_0>0$, $R\in(\rho,R_0]$ and $u\in L^\infty(\Rd\times I_R)$. Suppose that either
\begin{enumerate}[label=(\alph*)]
\item $s\geq 1/2$ and $u$ satisfies
\begin{align}\label{eqHoelder:EquationDrift}
\partial_tu-\Lambda|\nabla u|-\M_\rho^+u\leq A \enspace\text{and}\enspace \partial_tu+\Lambda|\nabla u|-\M_\rho^-u\geq -A \enspace\text{in}\enspace Q_R
\end{align}
in the viscosity sense or\label{itemHoelder:PartialHoelderAt0Drift}
\item $u$ satisfies
\begin{align}\label{eqHoelder:EquationNoDrift}
\partial_tu-\M_\rho^+u\leq A \enspace\text{and}\enspace \partial_tu-\M_\rho^-u\geq -A \enspace\text{in}\enspace Q_R
\end{align}
in the viscosity sense.\label{itemHoelder:PartialHoelderAt0NoDrift}
\end{enumerate}
Then there exist $\alpha\in(0,2s)$ and $C>0$ depending only on $d,s,\lambda$ and $\Lambda$ such that
\begin{align}\label{eqHoelder:Result}
R^\alpha\frac{|u(x,t)-u(0,0)|}{d((x,t),(0,0))^\alpha}\leq C\left(\|u\|_{L^\infty(\Rd\times I_R)}+R^{2s}A\right)
\end{align}
for any $(x,t)\in Q_R\setminus Q_\rho$. In the case \ref{itemHoelder:PartialHoelderAt0Drift}, $\alpha$ and $C$ also depend on $R_0$.
\end{theorem}

\noindent
In the context of Theorem \ref{thmIntro:InformalMainResult}, one would use $A=\norm{f}_{L^\infty}$.

Let us comment on the above result. First of all, it is not surprising that such an estimate holds for a fixed $R$ and $\rho$, since for any bounded function, it is true that
\begin{align*}
R^\alpha\frac{|u(x,t)-u(0,0)|}{d((x,t),(0,0))^\alpha}\leq R^\alpha\frac{2\|u\|_{L^\infty(Q_R)}}{\rho^\alpha}=2\left(\frac{R}{\rho}\right)^\alpha\|u\|_{L^\infty(Q_R)},
\end{align*}
for $(x,t)\in Q_R\setminus Q_\rho$ and any $\alpha\in(0,1)$, independently of the equation. However, this estimate blows up if $\rho\to 0^+$ or $R\to\infty$. The main point of the estimate \eqref{eqHoelder:Result} is the independence of $C$ and $\alpha$ with respect to $R$ and $\rho$.

It is standard in the non-divergence integro-differential regularity theory that the lower bound for $K\in\K_\rho$ is only required for small values of $\abs{y}$.  Once the lower bound holds for $\abs{y}$ small enough, one can modify the equation with a zero order term to artificially impose the lower bound for all $y$.  We make this note as a remark without proof.  See, for example, \cite[Section 14]{CaSi-09RegularityIntegroDiff}.

\begin{remark}
If one replaces the assumption \eqref{eqPre:LowerEllipticityAssumption} by
\begin{align*}
K(y)\geq\lambda\min\left\{|y|^{-d-2s},\rho^{-d-2s}\right\} \quad\text{for all}\ |y|<R_0
\end{align*}
for some $R_0>\rho$, then Theorem \ref{thmHoelder:PartialHoelderAt0} remains true for $R\in(\rho,R_0]$, where $\alpha$ and $C$ will depend also on $R_0$.
\end{remark}

\begin{remark}
Let us take a step back to general fully nonlinear equations of the form \eqref{eqPre:Equation} with $F$ as in \eqref{eqPre:FullyNonlinear}. It is immediate to check that any viscosity solution of \eqref{eqPre:Equation} automatically solves, in the viscosity sense, the corresponding extremal inequalities in \eqref{eqHoelder:EquationDrift} or \eqref{eqHoelder:EquationNoDrift}, and hence enjoys the partial Hölder regularity as in Theorem \ref{thmHoelder:PartialHoelderAt0}. As a particular case of \eqref{eqPre:Equation}, we cover linear equations, and thus, Theorem \ref{thmHoelder:PartialHoelderAt0} implies Theorem \ref{thmIntro:InformalMainResult}.
\end{remark}

Now we can move on to the proof of Theorem \ref{thmHoelder:PartialHoelderAt0}.  When one works with a family of equations that is invariant under the natural scaling, the typical proof (which is very standard by now) rescales the target function at each stage of the oscillation reduction to end up in a setting for which the version of Theorems \ref{thmHarnack:WeakHarnackDrift} or \ref{thmHarnack:WeakHarnackNoDrift} for $R=1$ is applicable.  If that were the case here, there would be no need to include the following proof.  We do include the proof as it includes all the relevant arguments that do not invoke any rescaling.  Another small deviation is that when $\rho>0$, the oscillation reduction argument contains only finitely many steps, in contrast to the case of $\rho=0$.

\begin{proof}[Proof of Theorem \ref{thmHoelder:PartialHoelderAt0}]
We will prove \eqref{eqHoelder:Result} only for case \ref{itemHoelder:PartialHoelderAt0Drift}. For part \ref{itemHoelder:PartialHoelderAt0NoDrift}, the proof only differs by neglecting the drift term and invoking part \ref{itemHarnack:CorHarnackNoDrift} of Corollary \ref{corHarnack:HarnackCorollary} instead of part \ref{itemHarnack:CorHarnackDrift}. Note that inequality \eqref{eqHoelder:Result} holds trivially if $u\equiv 0$ in $\R^d\times I_R$. Hence, we assume that $\|u\|_{L^\infty(\Rd\times I_R)}>0$. Furthermore, we claim that it suffices to prove 
\begin{align}\label{eqHoelder:WlogToProve}
R^\alpha\frac{|u(x,t)-u(0,0)|}{d((x,t),(0,0))^\alpha}\leq C
\end{align}
for all $u$ satisfying
\begin{align}\label{eqHoelder:WlogAssumption}
\|u\|_{L^\infty(\Rd\times I_R)}\leq\frac{1}{2} \quad\text{and}\quad \left(\frac{R}{2}\right)^{2s}A\leq\frac{\varepsilon}{4}
\end{align}
for some $\varepsilon>0$ that will be determined later. Indeed, by replacing $u$ with
\begin{align*}
\frac{u}{2\|u\|_{L^\infty(\R^d\times I_R)}+\frac{4}{\varepsilon}\left(\frac{R}{2}\right)^{2s}A},
\end{align*}
we achieve \eqref{eqHoelder:WlogAssumption} and once \eqref{eqHoelder:WlogToProve} is proven, \eqref{eqHoelder:Result} will follow for a possibly different constant $C$. Let $N$ be the integer satisfying
\begin{align}\label{eqHoelder:RangeN}
4^{-(N+1)}R<\rho\leq 4^{-N}R. 
\end{align}
If $\rho=0$, we set $N=\infty$. It now suffices to find a nondecreasing sequence $\{m_k\}_{k=0}^N$ and a nonincreasing sequence $\{M_k\}_{k=0}^N$ that fulfill
\begin{align}\label{eqHoelder:ToProve}
m_k\leq u\leq M_k \enspace\text{in}\enspace Q_{4^{-k}R} \quad\text{and}\quad M_k-m_k=4^{-k\alpha}
\end{align}
for any $k\in\{0,\ldots,N\}$ and some $\alpha\in(0,2s)$ that will be chosen below. Indeed, for any $(x,t)\in Q_R\setminus Q_\rho$, there exists $k\in\{0,\ldots,N\}$ such that $4^{-(k+1)}R<d((x,t),(0,0))\leq 4^{-k}R$, and then it would follow from \eqref{eqHoelder:ToProve} that
\begin{align*}
R^\alpha\frac{|u(x,t)-u(0,0)|}{d((x,t),(0,0))^\alpha}\leq R^\alpha\frac{\sup_{Q_{4^{-k}R}}u-\inf_{Q_{4^{-k}R}}u}{\left(4^{-(k+1)}R\right)^\alpha}\leq 4^{(k+1)\alpha}(M_k-m_k)=4^\alpha,
\end{align*}
which, under \eqref{eqHoelder:WlogAssumption}, is equivalent to \eqref{eqHoelder:WlogToProve} for $C=4^\alpha$. Let us construct $\{m_k\}_k$ and $\{M_k\}_k$ by finite induction if $\rho>0$ and infinite induction if $\rho=0$. First, we set $m_0=-\frac{1}{2}$ and $M_0=\frac{1}{2}$ such that \eqref{eqHoelder:ToProve} is satisfied for $k=0$ due to \eqref{eqHoelder:WlogAssumption}. Suppose now that we have defined the sequences up to $k\leq N-1$ and consider the function
\begin{align*}
v\coloneqq 2\frac{u-m_k}{M_k-m_k} \quad\text{in}\enspace\R^d\times I_R.
\end{align*}
By the induction hypothesis, we see that $0\leq v\leq 2$ in $Q_{4^{-k}R}$. At least one of the inequalities
\begin{align}\label{eqHoelder:Poss1}
\left|\{v\geq 1\}\cap\left(B_{4^{-(k+1)}R}\times(I_{4^{-k}R/2}\setminus I_{4^{-(k+1)}R})\right)\right|\geq\frac{1}{2}\left|B_{4^{-(k+1)}R}\times(I_{4^{-k}R/2}\setminus I_{4^{-(k+1)}R})\right|
\end{align}
and
\begin{align}\label{eqHoelder:Poss2}
\left|\{v\leq 1\}\cap\left(B_{4^{-(k+1)}R}\times(I_{4^{-k}R/2}\setminus I_{4^{-(k+1)}R})\right)\right|\geq\frac{1}{2}\left|B_{4^{-(k+1)}R}\times(I_{4^{-k}R/2}\setminus I_{4^{-(k+1)}R})\right|
\end{align}
has to hold true. Let us assume that \eqref{eqHoelder:Poss1} is in force. Otherwise, if \eqref{eqHoelder:Poss2} is true, then we can repeat the same argument for $2-v$ instead of $v$. Note that $v$ satisfies
\begin{align}\label{eqHoelder:EquationForV}
\partial_tv+\Lambda|\nabla v|-\M_\rho^-v\geq -\frac{2}{M_k-m_k}A \enspace\text{in}\enspace Q_R
\end{align}
in the viscosity sense. We want to use Corollary \ref{corHarnack:HarnackCorollary} for the function $w\coloneqq v_+$, which coincides with $v$ in $Q_{4^{-k}R}$. Let us claim that $w$ solves
\begin{align}\label{eqHoelder:ClaimForW}
\partial_tw+\Lambda|\nabla w|-\M_\rho^-w\geq -\left(\frac{4^{-k}R}{2}\right)^{-2s}\varepsilon\enspace\text{in}\enspace Q_{4^{-k}R/2}
\end{align}
in the viscosity sense. Observe that $\partial_tw=\partial_tv$ and $\nabla w=\nabla v$ in $Q_{4^{-k}R}$. Furthermore, we can apply \eqref{eqPre:FullyNonlinear} for $F=\M_\rho^-$, namely
\begin{align*}
\M_\rho^-w=\M_\rho^-(v-(-v_-))\leq\M_\rho^-v-\M_\rho^-(-v_-)=\M_\rho^-v+\M_\rho^+v_-.
\end{align*}
Let us notice that
\begin{align*}
\frac{2}{M_k-m_k}A=2\cdot 4^{k\alpha}A=\left(\frac{4^{-k}R}{2}\right)^{-2s}\cdot 4^{k(\alpha-2s)}\cdot 2\left(\frac{R}{2}\right)^{2s}A\leq\left(\frac{4^{-k}R}{2}\right)^{-2s}\frac{\varepsilon}{2}
\end{align*}
by \eqref{eqHoelder:WlogAssumption}, since we imposed $\alpha<2s$. Therefore, \eqref{eqHoelder:EquationForV} yields that it is now sufficient to prove the tail estimate
\begin{align}\label{eqHoelder:ClaimNegativePart}
\M_\rho^+v_-\leq\left(\frac{4^{-k}R}{2}\right)^{-2s}\frac{\varepsilon}{2}\enspace\text{in}\enspace Q_{4^{-k}R/2}.
\end{align}
This can be checked pointwise, and the inequality will remain true also in the viscosity sense. For that, fix $K\in\K_\rho$ and let us claim that
\begin{align}\label{eqHoelder:TailBound}
v(y,t)\geq -2\left(4^{(k+1)\alpha}\left(\frac{|y|}{R}\right)^\alpha-1\right)
\end{align}
for any $y\in\R^d\setminus B_{4^{-k}R}$ and $t\in I_R$. Indeed, for $j\in\{0,\ldots,k-1\}$ and $y\in B_{4^{-j}R}\setminus B_{4^{-(j+1)}R}$, we have
\begin{align*}
v(y,t)=2\frac{u(y,t)-M_k+M_k-m_k}{M_k-m_k}&\geq 2\left(\frac{m_j-M_j}{M_k-m_k}+1\right) \\
&=-2(4^{(k-j)\alpha}-1)\geq -2\left(4^{(k+1)\alpha}\left(\frac{|y|}{R}\right)^\alpha-1\right)
\end{align*}
by induction hypothesis, whereas for $y\in\R^d\setminus B_R$, we see that
\begin{align*}
v(y,t)=2\frac{u(y,t)-M_k+M_k-m_k}{M_k-m_k}&\geq 2\left(\frac{-1/2-1/2}{4^{-k\alpha}}+1\right) \\
&=-2(4^{k\alpha}-1)\geq -2\left(4^{(k+1)\alpha}\left(\frac{|y|}{R}\right)^\alpha-1\right)
\end{align*}
by \eqref{eqHoelder:WlogAssumption}. Since $v\geq 0$ in $Q_{4^{-k}R}$, we find that, for fixed $(x,t)\in Q_{4^{-k}R/2}$,
\begin{align}\label{eqHoelder:TailEstimate}
L_Kv_-(x,t)=\intd v_-(x+y,t)K(y)\d y\leq 2\int_{\Rd\setminus B_{4^{-k}R}}\left(4^{(k+1)\alpha}\left(\frac{|y|}{R}\right)^\alpha-1\right)K(y-x)\d y
\end{align}
by \eqref{eqHoelder:TailBound}. Let us remark here that $v_-=0$ in a neighborhood of $x$. In particular, $\nabla v_-(x)$ vanishes. To this end, notice that, for $y\in\Rd\setminus B_{4^{-k}R}$, $|y-x|\geq |y|-|x|\geq 4^{-k}R/2$ and therefore,
\begin{align}\label{eqHoelder:DistanceEstimate}
\frac{|y|}{|y-x|}\leq\frac{|x|+|y-x|}{|y-x|}=\frac{|x|}{|y-x|}+1\leq\frac{4^{-k}R/2}{4^{-k}R/2}+1=2.
\end{align}
It now follows from \eqref{eqHoelder:TailEstimate} and \eqref{eqHoelder:DistanceEstimate} that
\begin{align*}
L_Kv_-(x,t)&\leq 2\int_{\Rd\setminus B_{4^{-k}R}}\left(4^{(k+1)\alpha}\left(\frac{2|y-x|}{R}\right)^\alpha-1\right)K(y-x)\d y \\
&\leq 2\int_{\Rd\setminus B_{4^{-k}R/2}}\left(4^{(k+1)\alpha}\left(\frac{2|y|}{R}\right)^\alpha-1\right)K(y)\d y,
\end{align*}
where we also used the fact that $B_{4^{-k}R}(x)\supseteq B_{4^{-k}R/2}$ in the last step. Let us split the integration domain into appropriate annuli and use \eqref{eqPre:UpperEllipticityAssumption}. We arrive at
\begin{align*}
L_Kv_-(x,t)&\leq 2\sum_{i=0}^\infty\int_{B_{2^{i-2k}R}\setminus B_{2^{i-2k-1}R}}\left(4^{(k+1)\alpha}\left(\frac{2|y|}{R}\right)^\alpha-1\right)K(y)\d y \\
&\leq 2\sum_{i=0}^\infty\left(4^{(k+1)\alpha}\left(\frac{2\cdot 2^{i-2k}R}{R}\right)^\alpha-1\right)\int_{B_{2^{i-2k}R}\setminus B_{2^{i-2k-1}R}}K(y)\d y \\
&\leq 2\sum_{i=0}^\infty\left(4^{(k+1)\alpha}\cdot 2^{(i-2k+1)\alpha}-1\right)\Lambda (2^{i-2k-1}R)^{-2s} \\
&=2\Lambda\sum_{i=0}^\infty\left(2^{(i+3)\alpha}-1\right)2^{-2si}\left(2^{-2k-1}R\right)^{-2s} \\
&=2\Lambda\left(8^\alpha\sum_{i=0}^\infty 2^{-i(2s-\alpha)}-\sum_{i=0}^\infty 2^{-2si}\right)\left(4^{-k}\frac{R}{2}\right)^{-2s} \\
&=2\Lambda\left(\frac{8^\alpha}{1-2^{-(2s-\alpha)}}-\frac{1}{1-2^{-2s}}\right)\left(4^{-k}\frac{R}{2}\right)^{-2s}
\end{align*}
for any $\alpha\in(0,2s)$. We are left to choose $\alpha=\alpha(s,\Lambda,\varepsilon)\in(0,2s)$ small enough so that
\begin{align*}
2\Lambda\left(\frac{8^\alpha}{1-2^{-(2s-\alpha)}}-\frac{1}{1-2^{-2s}}\right)\leq\frac{\varepsilon}{2},
\end{align*}
whence \eqref{eqHoelder:ClaimNegativePart} will follow by taking supremum with respect to $K\in\K_\rho$. Recall that we have shown \eqref{eqHoelder:ClaimForW} in order to use Corollary \ref{corHarnack:HarnackCorollary}, which gives
\begin{align*}
\inf_{Q_{4^{-(k+1)}R}}w\geq c\fint_{I_{4^{-k}R/2}\setminus I_{4^{-(k+1)}R}}\fint_{B_{4^{-k}R/2}}w(x,t)\d x\d t-\varepsilon
\end{align*}
for some $c=c(d,s,\lambda,\Lambda,R_0)>0$. Note that Corollary \ref{corHarnack:HarnackCorollary} was indeed applicable in $Q_{4^{-k}R/2}$ even if $\rho>0$ because $4^{-k}R/2\geq 4^{-(N-1)}R/2\geq 4^{-N}R\geq\rho$ due to \eqref{eqHoelder:RangeN}. Since $v$ and $w$ coincide on the integration domain, we see from \eqref{eqHoelder:Poss1} that
\begin{align*}
\inf_{Q_{4^{-(k+1)}R}}v\geq\frac{c}{2}-\varepsilon\geq\varepsilon,
\end{align*}
provided we have chosen $\varepsilon=\min\{c/4,1\}=\varepsilon(d,s,\lambda,\Lambda,R_0)\in(0,1]$. In other words, we found that
\begin{align*}
u\geq m_k+4^{-k\alpha}\frac{\varepsilon}{2}\enspace\text{in}\enspace Q_{4^{-(k+1)}R}
\end{align*}
by the very definition of $v$. Let us finally define $M_{k+1}\coloneqq M_k$ and $m_{k+1}\coloneqq m_k+4^{-k\alpha}(1-4^{-\alpha})$ such that $M_{k+1}-m_{k+1}=4^{-(k+1)\alpha}$ and $u\leq M_{k+1}$ in $Q_{4^{-(k+1)}R}$. If we take
\begin{align*}
\alpha\leq -\frac{\log\left(1-\frac{\varepsilon}{2}\right)}{\log 4},
\end{align*}
then we indeed get
\begin{align*}
u\geq m_k+4^{-k\alpha}\frac{\varepsilon}{2}=m_{k+1}-4^{-k\alpha}\left(1-4^{-\alpha}\right)+4^{-k\alpha}\frac{\varepsilon}{2}=m_{k+1}+4^{-k\alpha}\left(4^{-\alpha}-\left(1-\frac{\varepsilon}{2}\right)\right)\geq m_{k+1}
\end{align*}
in $Q_{4^{-(k+1)}R}$, which needs to be proven. Note that $\alpha$ now depends only on $d,s,\lambda, \Lambda$ and $R_0$ as desired.
\end{proof}

While the above theorem provides partial Hölder regularity at the origin, the following corollary looks like a usual Hölder regularity result except for the restriction $d((x,t),(y,\tau))>\rho$.

\begin{corollary}\label{corHoelder:CorollaryHoelder}
Let the assumptions of Theorem \ref{thmHoelder:PartialHoelderAt0} be satisfied for $R>2\rho$. Then there exist $\alpha\in(0,2s)$ and $C>0$ depending only on $d,s,\lambda,\Lambda$ and $R_0$ such that
\begin{align*}
\sup_{\substack{(x,t),(y,\tau)\in Q_{R/4}\\d((x,t),(y,\tau))>\rho}}\frac{|u(x,t)-u(y,\tau)|}{d((x,t),(y,\tau))^\alpha}\leq\frac{C}{R^\alpha}\left(\|u\|_{L^\infty(\Rd\times I_R)}+R^{2s}A\right).
\end{align*}
In the case \ref{itemHoelder:PartialHoelderAt0NoDrift} of Theorem \ref{thmHoelder:PartialHoelderAt0}, $\alpha$ and $C$ do not depend on $R_0$.
\end{corollary}

We note the basic observation that $R>2\rho$ is intentional, as Theorem \ref{thmHoelder:PartialHoelderAt0} is eventually invoked in $Q_{R/2}$, whereby we will have an estimate for points in $Q_{R/2}\setminus Q_\rho$.  For completeness, we include the straightforward proof in Appendix \ref{secApp:Proofs}.

As a byproduct, we obtain Liouville's theorem.

\begin{proposition}[Liouville's theorem]
Let $u\in L^\infty(\Rd\times(-\infty,0])$ satisfy
\begin{align*}
\partial_tu-\M_\rho^+u\leq 0 \enspace\text{and}\enspace \partial_tu-\M_\rho^-u\geq 0 \enspace\text{in}\enspace\Rd\times(-\infty,0]
\end{align*}
in the viscosity sense.
Then $u$ is constant in $\Rd\times(-\infty,0]$.
\end{proposition}

\begin{proof}
Since $u$ satisfies the above inequalities on any cylinder $Q_R$ with $R>0$, we can apply Corollary \ref{corHoelder:CorollaryHoelder} for any $R>2\rho$ and $A=0$. Hence, there exists $\alpha\in(0,2s)$ and $C>0$ that depend only on $d,s,\lambda$ and $\Lambda$ such that
\begin{align*}
\sup_{\substack{(x,t),(y,\tau)\in Q_{R/4}\\d((x,t),(y,\tau))>\rho}}\frac{|u(x,t)-u(y,\tau)|}{d((x,t),(y,\tau))^\alpha}\leq\frac{C}{R^\alpha}\|u\|_{L^\infty(\Rd\times I_R)}\leq\frac{C}{R^\alpha}\|u\|_{L^\infty(\Rd\times(-\infty,0])}.
\end{align*}
Note that, in particular, they do not depend on an upper bound for $R$, since the drift term is not present. Letting $R\to\infty$, we see that
\begin{align*}
\sup_{\substack{(x,t),(y,\tau)\in\Rd\times(-\infty,0]\\d((x,t),(y,\tau))>\rho}}\frac{|u(x,t)-u(y,\tau)|}{d((x,t),(y,\tau))^\alpha}=0,
\end{align*}
which proves the assertion.
\end{proof}

\begin{remark}
Note that $\rho$ can be arbitrarily large. This means that Liouville's theorem is not a consequence of the integro-differential structure of the operator, but of its difference structure.
\end{remark}


\section{Related results: approximation via integrable kernels and elliptic problems }

Although the primary goal of our work is to demonstrate how solutions of integro-differential equations with integrable interaction kernels inherit regularity up to the scale where the kernel is large enough, there is a closely related question of approximating equations whose interaction kernels have a singularity at the origin by those which are integrable. Our work implies new results for such approximations.  Furthermore, for elliptic problems with integrable kernels, there is no requirement for any regularity past boundedness of $u$ for the relevant operators to be defined, and we explain how our result applies in this situation.  Both issues of approximation and classical pointwise solutions of elliptic problems have been addressed earlier, in \cite{FelmerTopp-2013ConvergenceZeroOrderToNonlinearFracHeat-IsraelJournMath}, \cite{FelmerTopp-2015UnifEquicontinuityZeroOrderEqApproxFracLaplace-CPDE} and \cite{FelmerDosPrazeresTopp-2018InteriorRegZeroOrderOpsFracLaplace-IsrealJournMath}.
One could view the results of this section as a generalization of those combined between \cite{FelmerTopp-2013ConvergenceZeroOrderToNonlinearFracHeat-IsraelJournMath} and \cite{FelmerDosPrazeresTopp-2018InteriorRegZeroOrderOpsFracLaplace-IsrealJournMath}.  In the former of those works, it was proved that one can approximate solutions of an equation with kernels with singularity by those with integrable kernels, and in the latter of those works, it was shown that for the specific case of a linear Poisson equation for the fractional Laplacian, one can obtain partial regularity results exactly as in a elliptic case of Theorem \ref{thmIntro:InformalMainResult} that carry over to the regularity results expected of the original equation.


\subsection{Approximation of solutions via truncated kernels}\label{subsec:Approximation}

Consider a Cauchy (terminal value) problem given by
\begin{align}\label{eqApproximation:Cauchy}
\begin{cases}
	\partial_tu-F(u)=f&\, \text{in}\ \Rd\times(-T,0], \\
	u(\cdot,0)=u_0&\, \text{in}\ \Rd,
\end{cases}
\end{align}
where $T>0$, $f\colon\Rd\times(-T,0]\to\R$ and $u_0\colon\Rd\to\R$ are bounded and uniformly continuous and $F$ is a fully nonlinear operator \eqref{eqPre:Isaac} with respect to kernels $K^{\alpha,\beta}\in\K_0$ (i.e.\ non-integrable kernels for $\rho=0$) and simply without any drifts, that is, $b^{\alpha,\beta}\equiv 0$. For the sake of brevity, we simply \emph{assume} that $F$ is such that \eqref{eqApproximation:Cauchy} admits a unique viscosity solution. We note that if $F$ is translation invariant, then the uniqueness is immediate, but if $F$ is not translation invariant, this is a non-trivial assumption, which is connected to fundamental open questions for existence and uniqueness for solutions of such equations.

We want to approximate solutions of \eqref{eqApproximation:Cauchy} in the following way. Choose a sequence $\rho_n\to 0^+$ and consider problems
\begin{align}\label{eqApproximation:CauchyTruncated}
\begin{cases}
	\partial_tu-F_n(u)=f&\, \text{in}\ \Rd\times(-T,0], \\
	u(\cdot,0)=u_0&\, \text{in}\ \Rd,
\end{cases}
\end{align}
where $F_n$ is constructed as follows. For any $\alpha,\beta$, define
\begin{align*}
K_n^{\alpha,\beta}(y)\coloneqq\min\left\{K^{\alpha,\beta}(y),\frac{\lambda}{\rho_n^{d+2s}}\right\},\quad y\in\Rd.
\end{align*}
Then $K_n^{\alpha,\beta}\in\K_{\rho_n}$ such that $F_n$, defined as in \eqref{eqPre:Isaac} with respect to the family $\{K_n^{\alpha,\beta}\}_{\alpha,\beta}$, is a fully nonlinear operator with respect to $\K_{\rho_n}$.

Although we do not provide proofs here, if the operators, $F$ and $F_n$, are translation invariant and obey their respective ellipticity classes, one can show that there exist unique viscosity solutions for these equations that obey a comparison principle and stability.  Under slightly different assumptions on $\K_\rho$, these properties and other results were studied in detail in \cite{FelmerTopp-2013ConvergenceZeroOrderToNonlinearFracHeat-IsraelJournMath}.  They show that this approximation scheme produces uniformly continuous solutions, $u_n$, of \eqref{eqApproximation:CauchyTruncated} that converge to the solution, $u$, of the original equation \eqref{eqApproximation:Cauchy}. In the framework of \cite{FelmerTopp-2013ConvergenceZeroOrderToNonlinearFracHeat-IsraelJournMath}, the uniform continuity originates from the modulus of $u_0$ and $f$.

In the context of Corollary \ref{corHoelder:CorollaryHoelder}, one can say more about the regularity of $u_n$. Indeed, $u_n$ enjoys the Hölder estimate
\begin{align*}
\frac{|u_n(x,t)-u_n(y,\tau)|}{d((x,t),(y,\tau))^\alpha}\leq C\left(\|u_n\|_{L^\infty(\Rd\times(-T,0])}+R^{2s}\|f\|_{L^\infty(\Rd\times(-T,0])}\right)
\end{align*}
for any $(x,t)\neq (y,\tau)\in Q_R$ for any $R>0$ such that $I_{4R}\subseteq(-T,0]$. Here, $C$ and $\alpha$ depend only on $d,s,\lambda$ and $\Lambda$.

Thus, by uniform convergence and stability, the partial Hölder regularity that $u_n$ inherits from $\K_{\rho_n}$ can be propagated back to $u$. In this setting, the regularity for $u$ is not in question, as it was already contained in \cite{Sil14}, but rather the interesting part is that the approximating $u_n$ have uniform regularity.

Finally, it is also worth noting a result from \cite{Cha07} that shows the following linear mixed problems not only have continuous solutions, but they have nice behavior in time.  To this end, consider
\begin{align*}
\begin{cases}
	\partial_tu=L_Ku&\,\text{in}\ \Omega\times(-\infty,0], \\
	u=g&\,\text{in}\ ((\Rd\setminus\Omega)\times(-\infty,0])\cup(\Omega\times\{0\}),
\end{cases}
\end{align*}
with some bounded domain $\Omega\subseteq\Rd$, $K\in\K_\rho\cap L^1(\Rd)$, $\rho>0$ and some function $g\in C_b$.  These equations admit unique distributional solutions $u\in C(\overline{\Omega}\times(-\infty,0])$ according to \cite[Theorem 3.1]{Cha07}. The distributional sense means that, for any $\varphi\in C_c^\infty((-\infty,0])$,
\begin{align*}
\int_{-\infty}^0 u(x,t)\partial_t\varphi(t)\d t=\int_{-\infty}^0L_Ku(x,t)\varphi(t)\d t.
\end{align*}
Since $L_Ku$ is continuous due to dominated convergence theorem, we find that
\begin{align*}
u(x,t)=u(x,0)-\int_t^0 L_Ku(x,\tau)\d\tau
\end{align*}
such that $u(x,\cdot)\in C^1((-\infty,0])$ for any $x\in\Rd$ and the equation holds classically. However, the latter result is not true in general for $K\notin L^1(\Rd)$ as solutions converge locally uniformly, but not in $C^1$ in time.


\subsection{The elliptic case}\label{subsec:Elliptic}

In the special case of elliptic problems, there is some interesting behavior that is a consequence of the fact that, for integrable $K$, the equation can be evaluated pointwise for $L^\infty$-functions.  To this end, we investigate viscosity solutions, $u\colon\Rd\to\R$, of
\begin{align*}
F(u)=f \enspace\text{in}\enspace B_R
\end{align*}
for some given function $f\colon B_R\to\R$. As before, $F$ is a fully nonlinear operator with respect to $\K_\rho$. The notion of viscosity solutions is defined in the same manner as in Definition \ref{defPre:Viscosity}. For any solution $u$, we can consider it as a function $u\colon\Rd\times I_R\to\R$, which is constant in time.  It turns out that partial regularity for integrable kernels was already demonstrated for a special case of a linear equation in \cite{FelmerDosPrazeresTopp-2018InteriorRegZeroOrderOpsFracLaplace-IsrealJournMath}. In this way, we can apply our parabolic result to obtain the following analogous result for the elliptic case, which generalizes \cite[Theorem 1.1]{FelmerDosPrazeresTopp-2018InteriorRegZeroOrderOpsFracLaplace-IsrealJournMath}, and we omit the proof.

\begin{theorem}\label{thmElliptic:Application}
Let $A\geq 0$, $R_0>0$, $R\in(2\rho,R_0]$ and $u\in L^\infty(\Rd)$. Suppose that either
\begin{enumerate}[label=(\alph*)]
\item $s\geq 1/2$ and $u$ satisfies
\begin{align*}
\Lambda|\nabla u|+\M_\rho^+u\geq -A \enspace\text{and}\enspace \Lambda|\nabla u|-\M_\rho^-u\geq -A \enspace\text{in}\enspace B_R
\end{align*}
in the viscosity sense or\label{itemElliptic:Drift}
\item $u$ satisfies
\begin{align*}
\M_\rho^+u\geq -A \enspace\text{and}\enspace \M_\rho^-u\leq A \enspace\text{in}\enspace B_R
\end{align*}
in the viscosity sense.
\end{enumerate}
Then there exist $\alpha\in(0,2s)$ and $C>0$ depending only on $d,s,\lambda$ and $\Lambda$ such that
\begin{align}\label{eqElliptic:HoelderResult}
\sup_{\substack{x,y\in B_{R/4}\\|x-y|\geq\rho}}\frac{|u(x)-u(y)|}{|x-y|^\alpha}\leq\frac{C}{R^\alpha}\left(\|u\|_{L^\infty(\Rd\times I_R)}+R^{2s}A\right).
\end{align}
In the case \ref{itemElliptic:Drift}, $\alpha$ and $C$ also depend on $R_0$.
\end{theorem}

The elliptic case becomes of particular interest if the drift vanishes, and the kernels are integrable and symmetric. We define the set
\begin{align}\label{eqApproximation:SymmetricKernels}
\S_\rho\coloneqq\{K\in\K_\rho\mid K\in L^1(\real^d),\  K(y)=K(-y)\ \text{for almost all}\ y\in\Rd\}
\end{align}
of symmetric and integrable admissible kernels. 
Let us notice that when $K(x,\cdot)$ is integrable, the equation
\begin{align*}
\intd(u(x+y)+u(x-y)-2u(x))K(x,y)\d y=f(x),
\end{align*}
can be understood in a pointwise sense whenever $K(x,\cdot)\in\S_\rho$ with $\rho>0$ and $u\in L^\infty(\Rd)$. More precisely, any globally bounded function solves an equation with a bounded right hand side.  However, if a solution is merely bounded, then an issue arises in implementing our results, specifically in the proof of the weak Harnack inequality, where a supersolution must attain a minimum on a bounded set.  Thus, our result does not directly apply to pointwise solutions that are only bounded.  However, our result does apply to pointwise solutions that are continuous (or rather, the weak Harnack result applies for any bounded and lower semicontinuous supersolution).  This follows because continuous pointwise solutions are necessarily viscosity solutions.  Thus, we obtain the following result for continuous pointwise solutions. Note that we invoke $\S_\rho$ as the class of admissible kernels, defined in \eqref{eqApproximation:SymmetricKernels}, and importantly, these kernels are integrable. In this case, $L_Ku$ inherits the continuity of $u$ when $u$ is continuous and bounded, and $K$ is integrable.

\begin{corollary}\label{corApproximation:EllipticPointwise}
Let $A\geq 0$, $R>2\rho>0$ and $u\in C(B_R)\cap L^\infty(\Rd)$ satisfy
\begin{align}\label{eqApproximation:SymmetricEquationSup}
\sup_{K\in\S_\rho}\intd(u(x+y)+u(x-y)-2u(x))K(y)\d y\geq -A
\end{align}
and
\begin{align}\label{eqApproximation:SymmetricEquationInf}
\inf_{K\in\S_\rho}\intd(u(x+y)+u(x-y)-2u(x))K(y)\d y\leq A
\end{align}
pointwise for any $x\in B_R$. Then there exist $\alpha\in(0,2s)$ and $C>0$ depending only on $d,s,\lambda$ and $\Lambda$ such that \eqref{eqElliptic:HoelderResult} is fulfilled.
\end{corollary}

\begin{remark}
	As one will see, the continuity assumption does not appear directly in the proof below, but rather it is used implicitly with the requirement that viscosity subsolutions must be upper semicontinuous and viscosity supersolutions must be lower semicontinuous.
\end{remark}

\begin{proof}[Proof of Corollary \ref{corApproximation:EllipticPointwise}]
Note that under the additional symmetry assumption, the definitions of the extremal operators coincide with \eqref{eqPre:DefinitionPucci} on smooth functions, that is, all viscosity solutions coincide. In order to apply Theorem \ref{thmElliptic:Application}, it suffices to show that $u$ solves \eqref{eqApproximation:SymmetricEquationSup} and \eqref{eqApproximation:SymmetricEquationInf} also in the viscosity sense. Let us deduce that $u$ is a viscosity supersolution, and the subsolution property will follow similarly.  We first note that $u$, by assumption of continuity, is also lower semicontinuous in $B_R$.  Let $x_0\in B_R$ and suppose, there exists a test function $\varphi\colon B_\varepsilon(x_0)\to\R$ for some $\varepsilon>0$ such that $u(x_0)=\varphi(x_0)$ and $u\geq\varphi$ in $B_\varepsilon(x_0)$. Define $v$ as in \eqref{eqPre:ViscosityDefinitionConstructionV}. Fix $y\in B_\varepsilon$ and $K\in\S_\rho$. Then
\begin{align*}
\varphi(x_0+y)+\varphi(x_0-y)-2\varphi(x_0)&=(\varphi(x_0+y)-u(x_0))+(\varphi(x_0-y)-u(x_0)) \\
&=(\varphi(x_0+y)-u(x_0+y))+(\varphi(x_0-y)-u(x_0-y)) \\
&\quad +(u(x_0+y)+u(x_0-y)-2u(x)).
\end{align*}
Hence, using symmetry, we find that
\begin{align*}
&\inf_{K\in\S_\rho}\intd(v(x_0+y)+v(x_0-y)-2v(x_0))K(y)\d y \\
&\quad =\inf_{K\in\S_\rho}\left(\int_{B_\varepsilon}(\varphi(x_0+y)+\varphi(x_0-y)-2\varphi(x_0))K(y)\d y\right. \\
&\quad\quad\quad\quad\quad\quad\left.+\int_{\Rd\setminus B_\varepsilon}(u(x_0+y)+u(x_0-y)-2u(x_0))K(y)\d y\right) \\
&\quad =\inf_{K\in\S_\rho}\left(2\int_{B_\varepsilon}(\varphi(x_0+y)-u(x_0+y))K(y)\d y+\intd(u(x_0+y)+u(x_0-y)-2u(x_0))K(y)\d y\right) \\
&\quad\leq\inf_{K\in\S_\rho}\intd(u(x_0+y)+u(x_0-y)-2u(x_0))K(y)\d y\leq A
\end{align*}
and any of the above integrals exist due to integrability of $K\in\S_\rho$.
\end{proof}

\begin{remark}
As part of establishing the previous corollary, we proved that continuous pointwise solutions are viscosity solutions. In fact, the converse is also true, that any continuous viscosity solution will be a pointwise solution.  This direction relies on a slightly deeper observation that the set of points where a continuous function can have contact from above with a test function is dense (see, e.g.\ the proof of \cite[Lemma A.3]{GS16}).
\end{remark}

More results on convergence of solutions with respect to truncated kernels in the elliptic case can be found in \cite{FelmerTopp-2015UnifEquicontinuityZeroOrderEqApproxFracLaplace-CPDE}.


\appendix

\section{Additional Proofs}\label{secApp:Proofs}

Here, we collect proofs of various results and steps from above, which we felt were mostly straightforward and could be reasonably collected here.  For some of them, even though we are working in the case $\rho\geq 0$, they are identical to the corresponding proofs for $\rho=0$.

It will be convenient to state some properties of admissible kernels that follow from the taken assumptions \eqref{eqPre:LowerEllipticityAssumption}, \eqref{eqPre:UpperEllipticityAssumption} and \eqref{eqPre:SymmetryAssumption}, and we will only use those in the following proofs.

\begin{lemma} \label{lemmaPre:AssumptionFollowUp}
There exists some constant $C=C(s)>0$ such that, for any $K\in\K_\rho$, the following statements hold true:
\begin{enumerate}[label=(\alph*)]
\item For any $R>0$,
\begin{align*}
\int_{\Rd\setminus B_R}K(y)\d y\leq C\Lambda R^{-2s}.
\end{align*}\label{itemPre:AssumptionFollowUpAwayFromOrigin}
\item For $\alpha>2s$ and $R>0$,
\begin{align*}
\int_{B_R}|y|^\alpha K(y)\d y\leq C\frac{\Lambda}{1-2^{-(\alpha-2s)}}R^{\alpha-2s}.
\end{align*}\label{itemPre:AssumptionFollowUpCloseToOrigin}
\item For $s<1/2$ and $R>0$,
\begin{align*}
\int_{B_R}|y|K(y)\d y\leq C\Lambda R^{1-2s}.
\end{align*}\label{itemPre:AssumptionFollowUpCloseToOriginSmallS}
\item For $s>1/2$ and $R>0$,
\begin{align*}
\int_{\Rd\setminus B_R}|y|K(y)\d y\leq C\Lambda R^{1-2s}.
\end{align*}\label{itemPre:AssumptionFollowUpAwayFromOriginLargeS}
\item For any $R>0$,
\begin{align*}
\int_{B_R}|y|^2K(y)\d y\leq C\Lambda R^{2-2s}.
\end{align*}\label{itemPre:AssumptionFollowUpCloseToOriginSquare}
\item 
For $s=1/2$, $R>0$, $x\in\Rd$ and any measurable function $u\colon\Rd\to\R$ such that $L_Ku(x)$ is defined, we may write
\begin{align}\label{eqPre:AlternativeRepresentationOperator}
L_Ku(x)=\intd(u(x+y)-u(x)-y\cdot\nabla u(x)\Indicator_{B_R}(y))K(y)\d y,
\end{align}\label{itemPre:AlternativeRepresentationOperator}
which deviates from \eqref{eqPre:DefinitionDifferences}, \eqref{eqPre:Operator} when $R\not=1$, and the equality originates from \eqref{eqPre:SymmetryAssumption}.
\end{enumerate}
\end{lemma}

\begin{proof}
\begin{enumerate}[label=(\alph*)]
\item We split the integration domain in a way that allows to use \eqref{eqPre:UpperEllipticityAssumption}, that is,
\begin{align*}
\int_{\Rd\setminus B_R}K(y)\d y=\sum_{i=0}^\infty\int_{B_{2^{i+1}R}\setminus B_{2^iR}}K(y)\d y\leq\sum_{i=0}^\infty\Lambda(2^iR)^{-2s}=\frac{1}{1-2^{-2s}}\Lambda R^{-2s},
\end{align*}
which proves this part for $C\geq 1/(1-2^{-2s})$.
\item In this case, we split the integration domain in a different way. Note that $|y|$ is bounded on any annulus such that
\begin{align*}
\int_{B_R}|y|^\alpha K(y)\d y&=\sum_{i=0}^\infty\int_{B_{2^{-i}R}\setminus B_{2^{-i-1}R}}|y|^\alpha K(y)\d y\leq\sum_{i=0}^\infty\left(2^{-i}R\right)^\alpha\int_{B_{2^{-i}R}\setminus B_{2^{-i-1}R}}K(y)\d y \\
&\leq\sum_{i=0}^\infty\left(2^{-i}R\right)^\alpha\Lambda(2^{-i-1}R)^{-2s}=2^{2s}\sum_{i=0}^\infty 2^{-i(\alpha-2s)}\Lambda R^{\alpha-2s}=\frac{2^{2s}}{1-2^{-(\alpha-2s)}}\Lambda R^{\alpha-2s},
\end{align*}
which is finite, since $2s<\alpha$. Hence, this part is proven for an appropriate constant $C$.
\item Since $1>2s$, this is a particular case of part \ref{itemPre:AssumptionFollowUpCloseToOrigin} after adjusting the constant $C$.
\item Here, we can compute
\begin{align*}
\int_{\Rd\setminus B_R}|y|K(y)\d y&=\sum_{i=0}^\infty\int_{B_{2^{i+1}R}\setminus B_{2^iR}}|y|K(y)\d y\leq\sum_{i=0}^\infty 2^{i+1}R\int_{B_{2^{i+1}R}\setminus B_{2^iR}}K(y)\d y \\
&\leq\sum_{i=0}^\infty 2^{i+1}R\Lambda(2^iR)^{-2s}=2\sum_{i=0}^\infty 2^{i(1-2s)}\Lambda R^{1-2s}=\frac{2}{1-2^{1-2s}}\Lambda R^{1-2s}.
\end{align*}
\item This part is also a particular case of part \ref{itemPre:AssumptionFollowUpCloseToOrigin} after choosing the correct constant.
\item If $R\geq 1$, then the error we are making in \eqref{eqPre:AlternativeRepresentationOperator} is
\begin{align*}
\intd y\cdot\nabla u(x)\Indicator_{B_R\setminus B_1}(y)K(y)\d y=\int_{B_R\setminus B_1} yK(y)\d y\cdot\nabla u(x).
\end{align*}
Let $N\geq 1$ be the smallest integer such that $2^N\geq R$. Then the symmetry assumption \eqref{eqPre:SymmetryAssumption} yields
\begin{align*}
\left|\intd y\cdot\nabla u(x)\Indicator_{B_R\setminus B_1}(y)K(y)\d y\right|\leq\sum_{i=0}^{N-1}\left|\int_{B_{2^{i+1}}\setminus B_{2^i}} yK(y)\d y\right||\nabla u(x)|=0.
\end{align*}
In other words, we made no error and \eqref{eqPre:AlternativeRepresentationOperator} is correct. The case $R<1$ can be treated in the same way.
\end{enumerate}
\end{proof}

Note that in the definition of $\delta_y u$ in \eqref{eqPre:DefinitionDifferences} for $s=1/2$, the choice of $B_1$ seems artificial. Part \ref{itemPre:AlternativeRepresentationOperator} now allows to choose any radius $R$ for the indicator function in \eqref{eqPre:DefinitionDifferences}, and this will become important as we might choose a very specific value for $R$ in the upcoming proofs.

\begin{proof}[Proof of Lemma \ref{lemmaPre:GlobalTestFunction}]
Let $\varepsilon>0$ be arbitrary such that $B_\varepsilon(x_0)\times(t_0-\varepsilon,t_0]\subseteq\Omega\times I$. Define $v$ as in \eqref{eqPre:ViscosityDefinitionConstructionV}. By Definition \ref{defPre:Viscosity},
\begin{align*}
\partial_tv(x_0,t_0)+\Lambda|\nabla v(x_0,t_0)|-\M_\rho^-v(x_0,t_0)\geq f(x_0,t_0).
\end{align*}
Note that
\begin{align*}
\delta_yv(x_0,t_0)=v(x_0+y,t_0)-\varphi(x_0+y,t_0)+\delta_y\varphi(x_0,t_0)
\end{align*}
as $v(x_0,t_0)=u(x_0,t_0)=\varphi(x_0,t_0)$ and $\nabla v(x_0,t_0)=\nabla\varphi(x_0,t_0)$. Let us compute
\begin{align*}
\M_\rho^-v(x_0,t_0)&=\inf_{K\in\K_\rho}\intd\delta_yv(x_0,t_0)K(y)\d y \\
&=\inf_{K\in\K_\rho}\left(\intd(v(x_0+y,t_0)-\varphi(x_0+y,t_0))K(y)\d y+\intd\delta_y\varphi(x_0,t_0)K(y)\d y\right) \\
&\geq\inf_{K\in\K_\rho}\int_{\Rd\setminus B_\varepsilon}(u(x_0+y,t_0)-\varphi(x_0+y,t_0))K(y)\d y+\M_\rho^-\varphi(x_0,t_0) \\
&\geq\int_{\Rd\setminus B_\varepsilon}(u(x_0+y,t_0)-\varphi(x_0+y,t_0))\min\left\{\frac{\lambda}{|y|^{d+2s}},\frac{\lambda}{\rho^{d+2s}}\right\}\d y+\M_\rho^-\varphi(x_0,t_0),
\end{align*}
where in the last step we applied \eqref{eqPre:LowerEllipticityAssumption} to each kernel. Note that $\M_\rho^-\varphi(x_0,t_0)$ is well defined due to the additional assumption $\varphi(\cdot,t_0)\in L^\infty(\Rd)$. The other quantity on the right hand side is clearly nonnegative and bounded by $\M_\rho^-v(x_0,t_0)-\M_\rho^-\varphi(x_0,t_0)$, hence finite. Since $\partial_tv(x_0,t_0)=\partial_t\varphi(x_0,t_0)$ and $\nabla v(x_0,t_0)=\nabla\varphi(x_0,t_0)$, we obtain
\begin{multline*}
\partial_t\varphi(x_0,t_0)+\Lambda|\nabla\varphi(x_0,t_0)|-\M_\rho^-\varphi(x_0,t_0) \\
-\int_{\Rd\setminus B_\varepsilon}(u(x_0+y,t_0)-\varphi(x_0+y,t_0))\min\left\{\frac{\lambda}{|y|^{d+2s}},\frac{\lambda}{\rho^{d+2s}}\right\}\d y\geq f(x_0,t_0).
\end{multline*}
Since the integrand is nonnegative, we can apply monotone convergence theorem to arrive at \eqref{eqPre:GlobalTestFunction} by letting $\varepsilon\to 0^+$.
\end{proof}

\begin{proof}[Proof of Lemma \ref{lemmaHarnack:BumpFunctionOperatorBound}]
Fix $x\in\Rd$ and a kernel $K\in\K_\rho$. We have to work through different cases. Assume first that $s<1/2$. Using a Lipschitz and boundedness estimate, we see that
\begin{align*}
|\varphi(x+y)-\varphi(x)|\leq\min\left\{\|\nabla\varphi\|_{L^\infty(\Rd)}|y|,2\|\varphi\|_{L^\infty(\Rd)}\right\}\leq C\min\left\{\frac{|y|}{R},2\right\}.
\end{align*}
Hence, by Lemma \ref{lemmaPre:AssumptionFollowUp}\ref{itemPre:AssumptionFollowUpAwayFromOrigin} and \ref{itemPre:AssumptionFollowUpCloseToOriginSmallS}, there exists $C_1=C_1(s)>0$ such that
\begin{align*}
|L_K\varphi(x)|&\leq\frac{C}{R}\int_{B_R}|y|K(y)\d y+2C\int_{\Rd\setminus B_R}K(y)\d y \\
&\leq\frac{C}{R}C_1\Lambda R^{1-2s}+2CC_1\Lambda R^{-2s}=C_2R^{-2s}
\end{align*}
for $C_2=3CC_1\Lambda=C_2(s,\Lambda,C)>0$. By taking infimum or supremum over $K\in\K_\rho$ and supremum over $x\in\Rd$, we see that \eqref{eqHarnack:BumpFunctionOperator} holds for $\tilde{C}\geq C_2$.

Now let $s>1/2$ and perform a Taylor expansion, namely
\begin{align*}
\varphi(x)=\varphi(x+y)+y\cdot\nabla\varphi(x)+\frac{1}{2}\sum_{i,j=1}^d\partial_i\partial_j\varphi(\xi_y)y_iy_j,
\end{align*}
where $\xi_y\in\Rd$ is some point on the line between $x$ and $y$. If $y\in B_R$, we use this expansion to obtain
\begin{align*}
|\delta_y\varphi(x)|=|\varphi(x+y)-\varphi(x)-y\cdot\nabla\varphi(x)|\leq\frac{1}{2}\|D^2\varphi\|_{L^\infty(\Rd)}|y|^2\leq \frac{C}{2}\frac{|y|^2}{R^2}.
\end{align*}
On the other hand, if $y\in\Rd\setminus B_R$, note as before that $|\varphi(x+y)-\varphi(x)|\leq 2C$. Therefore, in view of Lemma \ref{lemmaPre:AssumptionFollowUp}\ref{itemPre:AssumptionFollowUpAwayFromOrigin}, \ref{itemPre:AssumptionFollowUpAwayFromOriginLargeS} and \ref{itemPre:AssumptionFollowUpCloseToOriginSquare}, we get that
\begin{align}
|L_K\varphi(x)|&\leq\frac{C}{2R^2}\int_{B_R}|y|^2K(y)\d y+2C\int_{\Rd\setminus B_R}K(y)\d y+\int_{\Rd\setminus B_R}|y|K(y)\d y|\nabla\varphi(x)| \label{eqHarnack:BoundOperatorComputationLargeS} \\
&\leq\frac{C}{2R^2}C_1\Lambda R^{2-2s}+2CC_1\Lambda R^{-2s}+C_1\Lambda R^{1-2s}\frac{C}{R}=C_3R^{-2s} \nonumber
\end{align}
for
\begin{align*}
C_3=\frac{7}{2}CC_1\Lambda=C_3(s,\Lambda,C)>0,
\end{align*}
whence the claim holds also in this case if $\tilde{C}\geq C_3$.

The remaining case $s=1/2$ can be treated in the same way. Observe that due to Lemma \ref{lemmaPre:AssumptionFollowUp}\ref{itemPre:AlternativeRepresentationOperator}, we can get rid of the last term in \eqref{eqHarnack:BoundOperatorComputationLargeS} such that all expression are finite. Hence, the claim \eqref{eqHarnack:BumpFunctionOperator} will hold true in any case for $\tilde{C}\coloneqq\max\{C_2,C_3\}=\tilde{C}(s,\Lambda,C)>0$.
\end{proof}

In order to prove Proposition \ref{propHarnack:HölderPucci}, we need the following simple lemma. We will omit the proof.

\begin{lemma}\label{lemmaHarnack:ModulusSupInf}
Let $\{u_\alpha\}_\alpha$ be an arbitrary family of uniformly continuous functions $u_\alpha\colon\Rd\to\R$ with a common modulus  of continuity $\omega\colon[0,\infty)\to[0,\infty)$. Then the functions
\begin{align*}
\sup_\alpha u_\alpha \and \inf_\alpha u_\alpha
\end{align*}
are also uniformly continuous on $\Rd$ with modulus $\omega$.
\end{lemma}

\begin{proof}[Proof of Proposition \ref{propHarnack:HölderPucci}]
In all cases, it suffices to prove the corresponding estimate for $L_K\varphi$, where $K\in K_\rho$ is any admissible kernel. The claims will then follow from Lemma \ref{lemmaHarnack:ModulusSupInf}. During the proof, fix $K\in\K_\rho$ and $x,y\in\Rd$ with $x\neq y$. Let $C_1=C_1(s)>0$ be the constant from Lemma \ref{lemmaPre:AssumptionFollowUp}.
\begin{enumerate}[label=(\alph*)]
\item Note first that
\begin{align*}
|L_K\varphi(x)-L_K\varphi(y)|&\leq\intd|(\varphi(x+z)-\varphi(x))-(\varphi(y+z)-\varphi(y))|K(z)\d z \\
&\leq\int_{B_{|x-y|}}(|\varphi(x+z)-\varphi(x)|+|\varphi(y+z)-\varphi(y)|)K(z)\d z \\
&\quad +\int_{\Rd\setminus B_{|x-y|}}(|\varphi(x+z)-\varphi(y+z)|+|\varphi(x)-\varphi(y)|)K(z)\d z.
\end{align*}
Let us estimate both terms separately. For the first integral, we use fact that $\varphi$ is Hölder- or Lipschitz continuous, respectively, as well as Lemma \ref{lemmaPre:AssumptionFollowUp}\ref{itemPre:AssumptionFollowUpCloseToOrigin}, that is,
\begin{align*}
\int_{B_{|x-y|}}(|\varphi(x+z)-\varphi(x)|+|\varphi(y+z)-\varphi(y)|)K(z)\d z&\leq 2[\varphi]_{C^{0,\alpha}(\Rd)}\int_{B_{|x-y|}}|z|^\alpha K(z)\d z \\
&\leq 2C_1\frac{\Lambda}{1-2^{-(\alpha-2s)}}[\varphi]_{C^{0,\alpha}(\Rd)}|x-y|^{\alpha-2s}.
\end{align*}
For the second term, the same approach gives
\begin{align*}
\int_{\Rd\setminus B_{|x-y|}}(|\varphi(x+z)-\varphi(y+z)|+|\varphi(x)-\varphi(y)|)K(z)\d z&\leq 2[\varphi]_{C^{0,\alpha}(\Rd)}|x-y|^\alpha\int_{\Rd\setminus B_{|x-y|}}K(z)\d z \\
&\leq 2C_1\Lambda[\varphi]_{C^{0,\alpha}(\Rd)}|x-y|^{\alpha-2s},
\end{align*}
but in this case we used part \ref{itemPre:AssumptionFollowUpAwayFromOrigin} of Lemma \ref{lemmaPre:AssumptionFollowUp}. Therefore, we have proven the first claim for an appropriate constant $C$.
\item By the fundamental theorem of calculus, we may write, for any $z\in\Rd$, 
\begin{align*}
\varphi(x+z)-\varphi(x)=\int_0^1\frac{\partial}{\partial r}\varphi(x+rz)\ dr=\int_0^1z\cdot\nabla\varphi(x+rz)\d r,
\end{align*}
whence
\begin{align*}
L_K\varphi(x)=\intd\delta_z\varphi(x)K(z)\d z&=\intd\left(\int_0^1 z\cdot\nabla\varphi(x+rz)\d r-z\cdot\nabla\varphi(x)\right)K(z)\d z \\
&=\intd\int_0^1 z\cdot(\nabla\varphi(x+rz)-\nabla\varphi(x))K(z)\d r\d z.
\end{align*}
From here on, we can proceed similarly as before. We have
\begin{align*}
|L_K\varphi(x)-L_K\varphi(y)|&\leq\intd\int_0^1|z\cdot((\nabla\varphi(x+rz)-\nabla\varphi(x))-(\nabla\varphi(y+rz)-\nabla\varphi(y)))|K(z)\d r\d z \\
&\leq\int_{B_{|x-y|}}\int_0^1|z|(|\nabla\varphi(x+rz)-\nabla\varphi(x)|+|\nabla\varphi(y+rz)-\nabla\varphi(y)|)K(z)\d r\d z \\
&\quad +\int_{\Rd\setminus B_{|x-y|}}\int_0^1|z|(|\nabla\varphi(x+rz)-\nabla\varphi(y+rz)|+|\nabla\varphi(x)-\nabla\varphi(y)|)K(z)\d r\d z.
\end{align*}
Now, we can use the regularity of $\nabla\varphi$, which gives
\begin{align*}
&\int_{B_{|x-y|}}\int_0^1|z|(|\nabla\varphi(x+rz)-\nabla\varphi(x)|+|\nabla\varphi(y+rz)-\nabla\varphi(y)|)K(z)\d z\d r \\
&\quad\leq 2[\varphi]_{C^{1,\alpha}(\Rd)}\int_{B_{|x-y|}}\int_0^1|z||rz|^\alpha K(z)\d r\d z \\
&\quad\leq 2[\varphi]_{C^{1,\alpha}(\Rd)}\int_{B_{|x-y|}}|z|^{1+\alpha}K(z)\d z\int_0^1 r^\alpha\d r\leq\frac{2C_1\Lambda}{(1+\alpha)\left(1-2^{-(1+\alpha-2s)}\right)}[\varphi]_{C^{1,\alpha}(\Rd)}|x-y|^{1+\alpha-2s}
\end{align*}
by Lemma \ref{lemmaPre:AssumptionFollowUp}\ref{itemPre:AssumptionFollowUpCloseToOrigin} and
\begin{align*}
&\int_{\Rd\setminus B_{|x-y|}}\int_0^1|z|(|\nabla\varphi(x+rz)-\nabla\varphi(y+rz)|+|\nabla\varphi(x)-\nabla\varphi(y)|)K(z)\d r\d z \\
&\quad\leq 2[\varphi]_{C^{1,\alpha}(\Rd)}|x-y|^\alpha\int_{\Rd\setminus B_{|x-y|}}|z|\d z\leq 2C_1\Lambda[\varphi]_{C^{1,\alpha}(\Rd)}|x-y|^{1+\alpha-2s}
\end{align*}
by Lemma \ref{lemmaPre:AssumptionFollowUp}\ref{itemPre:AssumptionFollowUpAwayFromOrigin}. Consequently, this part is proven for $C$ sufficiently large.
\item In this critical case, we will use an auxiliary result. By \cite[Lemma A.1.1]{FR24}, it suffices to prove that
\begin{align}\label{eqHarnack:HoelderPucciCriticalSToProve}
|L_K\varphi(x+y)+L_K\varphi(x-y)-2L_K\varphi(x)|\leq C_2[\varphi]_{C^{1,\alpha}(\Rd)}|y|^\alpha
\end{align}
for any $x,y\in\Rd$ and some $C_2=C_2(s,\Lambda,\alpha)>0$. First of all, due to Lemma \ref{lemmaPre:AssumptionFollowUp}\ref{itemPre:AlternativeRepresentationOperator}, we can write
\begin{align*}
L_K\varphi(x)&=\intd(\varphi(x+z)-\varphi(x)-z\cdot\nabla\varphi(x)\Indicator_{B_{|y|}}(z))K(z)\d z \\
&=\int_{B_{|y|}}\int_0^1 z\cdot(\nabla\varphi(x+rz)-\nabla\varphi(x))K(z)\d r\d z+\int_{\Rd\setminus B_{|y|}}(\varphi(x+z)-\varphi(x))K(z)\d z,
\end{align*}
where we rewrote the first integral in the same manner as in part \ref{itemHarnack:HölderPucciLargeS}. Outside $B_{|y|}$, we will use the estimate
\begin{align}\label{eqHarnack:HoelderPucciSecondOrderDifferences}
|\varphi(x+y)+\varphi(x-y)-2\varphi(x)|\leq\frac{2^\alpha}{1+\alpha}[\varphi]_{C^{1,\alpha}(\Rd)}|y|^{1+\alpha},
\end{align}
which we will first verify. By fundamental theorem of calculus, we have
\begin{align*}
\varphi(x+y)+\varphi(x-y)-2\varphi(x)&=(\varphi(x+y)-\varphi(x))+(\varphi(x-y)-\varphi(x)) \\
&=\int_0^1\frac{\partial}{\partial r}\varphi(x+ry)\d r+ \int_0^1\frac{\partial}{\partial r}\varphi(x-ry)\d r \\
&=\int_0^1 y\cdot(\nabla\varphi(x+ry)-\nabla\varphi(x-ry))\d r
\end{align*}
and consequently,
\begin{align*}
|\varphi(x+y)+\varphi(x-y)-2\varphi(x)|&\leq\int_0^1|y|[\varphi]_{C^{1,\alpha}(\Rd)}|2ry|^\alpha\d r \\
&=2^\alpha[\varphi]_{C^{1,\alpha}(\Rd)}|y|^{1+\alpha}\int_0^1 r^\alpha\d r=\frac{2^\alpha}{1+\alpha}[\varphi]_{C^{1,\alpha}(\Rd)}|y|^{1+\alpha},
\end{align*}
which proves \eqref{eqHarnack:HoelderPucciSecondOrderDifferences}. Now, we can organize the terms to estimate
\begin{align*}
&|L_K\varphi(x+y)+L_K\varphi(x-y)-2L_K\varphi(x)| \\
&\quad\leq\int_{B_{|y|}}\int_0^1|z||\nabla\varphi(x+y+rz)-\nabla\varphi(x+y)|K(z)\d r\d z \\
&\quad\quad +\int_{B_{|y|}}\int_0^1|z||\nabla\varphi(x-y+rz)-\nabla\varphi(x-y)|K(z)\d r\d z \\
&\quad\quad +2\int_{B_{|y|}}\int_0^1|z||\nabla\varphi(x+rz)-\nabla\varphi(x)|K(z)\d r\d z \\
&\quad\quad +\int_{\Rd\setminus B_{|y|}}|\varphi(x+z+y)+\varphi(x+z-y)-2\varphi(x+z)|K(z)\d z \\
&\quad\quad +\int_{\Rd\setminus B_{|y|}}|\varphi(x+y)+\varphi(x-y)-2\varphi(x)|K(z)\d z.
\end{align*}
For the first three terms, we will use similar computations as in part \ref{itemHarnack:HölderPucciLargeS} and for the last two terms, we apply \eqref{eqHarnack:HoelderPucciSecondOrderDifferences}. This will lead to
\begin{align*}
&|L_K\varphi(x+y)+L_K\varphi(x-y)-2L_K\varphi(x)| \\
&\quad\leq 4\int_{B_{|y|}}\int_0^1|z|[\varphi]_{C^{1,\alpha}(\Rd)}|rz|^\alpha K(z)\d r\d z+2\int_{\Rd\setminus B_{|y|}}\frac{2^\alpha}{1+\alpha}[\varphi]_{C^{1,\alpha}(\Rd)}|y|^{1+\alpha}K(z)\d z \\
&\quad\leq 2[\varphi]_{C^{1,\alpha}(\Rd)}\left(2\int_0^1 r^\alpha\d r\int_{B_{|y|}}|z|^{1+\alpha}K(z)\d z+\frac{2^\alpha}{1+\alpha}|y|^{1+\alpha}\int_{\Rd\setminus B_{|y|}}K(z)\d z\right) \\
&\quad\leq\frac{2}{1+\alpha}[\varphi]_{C^{1,\alpha}(\Rd)}\left(2C_1\frac{\Lambda}{1-2^{-(1+\alpha-1)}}|y|^{1+\alpha-1}+2^\alpha|y|^{1+\alpha}C_1\Lambda|y|^{-1}\right) \\
&\quad =\frac{2}{1+\alpha}C_1\Lambda\left(\frac{2}{1-2^{-\alpha}}+2^\alpha\right)[\varphi]_{C^{1,\alpha}(\Rd)}|y|^\alpha,
\end{align*}
where we have used Lemma \ref{lemmaPre:AssumptionFollowUp}\ref{itemPre:AssumptionFollowUpAwayFromOrigin} and \ref{itemPre:AssumptionFollowUpCloseToOrigin}. Hence, \eqref{eqHarnack:HoelderPucciCriticalSToProve} is proven for an appropriate constant $C_2$.
\end{enumerate}
\end{proof}

Moreover, in order to present the proof of Theorem \ref{thmHarnack:WeakHarnackDrift} in a complete manner, we justify the existence of a first crossing point as in \eqref{eqHarnack:FirstCrossingPoint} in detail. Let us recall the setting in the proof of Theorem \ref{thmHarnack:WeakHarnackDrift}.

\begin{lemma}\label{lemmaAppendix:JustificationFirstCrossingPoint}
Let $\varepsilon,R>0$ and $u:\Rd\times\overline{I}_R\to\R$ be nonnegative everywhere and lower semicontinuous in $Q_R$. Let $\varphi\in C(\Rd\times\overline{I}_R)$ be a function such that the following properties hold true:
\begin{enumerate}[label=(\roman*)]
\item $\varphi(x,-R^{2s})=-\varepsilon$ for all $x\in\Rd$,\label{itemAppendix:InitialValue}
\item $\varphi(x,t)\leq -\varepsilon$ for all $(x,t)\in(\Rd\setminus B_R)\times\overline{I}_R$,\label{itemAppendix:Support}
\item $u>\varphi$ outside $Q_R$,\label{itemAppendix:InequalityEverywhere}
\item $u<\varphi$ at some point in $Q_{R/2}$.\label{itemAppendix:InequalitySomewhere}
\end{enumerate}
Then there exist $\delta\geq 0$ and $(x_0,t_0)\in Q_R$ such that the function $\psi\coloneqq\varphi-\delta$ satisfies
\begin{align}\label{eqAppendix:FirstCrossingPoint}
u(x_0,t_0)=\psi(x_0,t_0) \quad\text{and}\quad u(x,t)\geq\psi(x,t) \enspace\text{for all}\enspace (x,t)\in\Rd\times(-R^{2s},t_0].
\end{align}
\end{lemma}

\begin{proof}
Define the function $v\coloneqq u-\varphi$, which is then lower semicontinuous in $Q_R$. Moreover, define the set
\begin{align*}
T\coloneqq\{t\in I_R\mid\exists x\in\Rd:v(x,t)<0\}.
\end{align*}
By \ref{itemAppendix:InequalitySomewhere}, $v<0$ at some point in $Q_{R/2}$ and we have $T\neq\emptyset$. Denote $t_0\coloneqq\inf T\in\overline{I}_R$. Let us first show that $t_0\neq -R^{2s}$. By \ref{itemAppendix:InitialValue}, we can find $\theta_1>0$ such that $\varphi(x,t)<0$ for any $(x,t)\in B_R\times[-R^{2s},-R^{2s}+\theta_1)$ due to locally uniform continuity of $\varphi$. If $t_0=-R^{2s}$, then we would find $t\in(-R^{2s},-R^{2s}+\theta_1)$ and $x\in\Rd$ with $v(x,t)<0$. By \ref{itemAppendix:InequalityEverywhere}, this can only happen when $x\in B_R$. Using the above observation, we arrive at the contradiction that $0\leq u(x,t)<\varphi(x,t)<0$. Hence, we found that $t_0\in I_R$. Now, choose a sequence $\{t_n\}_{n=1}^\infty\subseteq T$ such that $t_n\to t_0$ as $n\to\infty$. Denote by $\{x_n\}_{n=1}^\infty\subseteq\Rd$ the corresponding sequence with $v(x_n,t_n)<0$. As above, all $x_n$ have to lie in $B_R$, whence there exists a subsequence $\{x_{n_k}\}_{k=1}^\infty$ of $\{x_n\}_n$ that converges to some $x_0\in\overline{B}_R$. Assume for contradiction that $x_0\in\partial B_R$. Then \ref{itemAppendix:Support} gives $\varphi(x_0,t)\leq -\varepsilon<0$ for any $t\in\overline{I}_R$ such that, by continuity, there exists $\theta_2>0$ with $\varphi(x,t)<0$ for any $(x,t)\in B_{\theta_2}(x_0)\times\overline{I}_R$. In particular, for sufficiently large $k$, this implies $v(x_{n_k},t_{n_k})\geq -\varphi(x_{n_k},t_{n_k})>0$, which contradicts the definition of $t_{n_k}$. Therefore, $(x_0,t_0)\in Q_R$ and by lower semicontinuity, we obtain
\begin{align}\label{eqHarnack:LowerSemiContinuity}
v(x_0,t_0)\leq\liminf_{k\to\infty}v(x_{n_k},t_{n_k})\leq 0.
\end{align}

It is now left to choose $\delta\geq 0$ such that $\psi\coloneqq\varphi-\delta$ fulfills \eqref{eqAppendix:FirstCrossingPoint}. Let us define $\delta\coloneqq -v(x_0,t_0)\geq 0$. In fact, if $t_0\notin T$, then $v(x,t)\geq 0$ whenever $t\in(-R^{2s},t_0]$ and $x\in\Rd$. In particular, \eqref{eqHarnack:LowerSemiContinuity} tells that $v(x_0,t_0)=0$, i.e.\ $\delta=0$. Then $\psi=\varphi$ and \eqref{eqAppendix:FirstCrossingPoint} is immediately proven. On the other hand, if $t_0\in T$, we might have $v(x_0,t_0)<0$. Consider the function $w\coloneq u-\psi=v-v(x_0,t_0)$. We obtain $w(x_0,t_0)=0$ and for $(x,t)\in\Rd\times (-R^{2s},t_0)$, the definition of $t_0$ implies
\begin{align*}
w(x,t)=v(x,t)-v(x_0,t_0)\geq -v(x_0,t_0)=\delta>0.
\end{align*}
Recalling that $w=u-\psi$, we arrive at \eqref{eqAppendix:FirstCrossingPoint}.
\end{proof}

In the proof of Theorem \ref{thmHarnack:WeakHarnackDrift}, we applied Lemma \ref{lemmaAppendix:JustificationFirstCrossingPoint} to a function $\varphi$ as in \eqref{eqHarnack:DefinitionTest}, that is,
\begin{align*}
\varphi(x,t)\coloneqq m(t)\eta(x)-R^{-2s}\left(R^{2s}+t\right)a-\varepsilon,
\end{align*}
and the assumptions of Lemma \ref{lemmaAppendix:JustificationFirstCrossingPoint} are satisfied. Hence, there exist $\delta\geq 0$ and a first crossing point $(x_0,t_0)\in Q_R$ between $u$ and $\psi$, where
\begin{align*}
\psi(x,t)=\varphi(x,t)-\delta=m(t)\eta(x)-R^{-2s}\left(R^{2s}+t\right)a-\varepsilon-\delta.
\end{align*}
By setting $\tilde{\varepsilon}=\varepsilon+\delta>0$, we can identify $\varphi$ with $\psi$ if we change $\varepsilon$ to $\tilde{\varepsilon}$. Then \eqref{eqHarnack:FirstCrossingPoint} is indeed true for a function $\varphi$ of the form \eqref{eqHarnack:DefinitionTest}. Since we are only using \eqref{eqHarnack:DefinitionTest} and \eqref{eqHarnack:FirstCrossingPoint} from that point of the proof, the specific value of $\varepsilon$ is not important and the proof can be finished as presented above.

Let us also mention that this change of $\varepsilon$ is indeed necessary due to the lack of continuity of $u$. For a general lower semicontinuous function, it may happen that $t_0\in T$ and $v(x_0,t_0)>0$, but $v(x_0,t)<0$ for any $t<t_0$. Hence, $u(x_0,t_0)=\varphi(x_0,t_0)$ cannot be achieved due to a jump of $u$ at $(x_0,t_0)$. For a continuous function $u$, \eqref{eqAppendix:FirstCrossingPoint} can always be achieved with $\delta=0$.

Finally, we owe the proof of Corollary \ref{corHoelder:CorollaryHoelder}.

\begin{proof}[Proof of Corollary \ref{corHoelder:CorollaryHoelder}]
As before, we will only give the proof under the assumptions of Theorem \ref{thmHoelder:PartialHoelderAt0}\ref{itemHoelder:PartialHoelderAt0Drift}. Fix $(y,\tau)\in Q_{R/4}$ and define
\begin{align*}
v(z,r)\coloneqq u(y+z,\tau+r), \quad z\in\Rd,\ r\in I_{R/2}.
\end{align*}
For $(z,r)\in Q_{R/2}$, we get $(y+z,\tau+r)\in Q_R$ such that $v\in L^\infty(\Rd\times I_{R/2})$ solves
\begin{align*}
\partial_tv+b\cdot\nabla v-\M_\rho^+v\leq A \enspace\text{and}\enspace \partial_tv+b\cdot\nabla v-\M_\rho^-v\geq -A \enspace\text{in}\enspace Q_{R/2}
\end{align*}
in the viscosity sense. Therefore, Theorem \ref{thmHoelder:PartialHoelderAt0} yields
\begin{align*}
\left(\frac{R}{2}\right)^\alpha\frac{|v(z,r)-v(0,0)|}{d((z,r),(0,0))^\alpha}\leq \tilde{C}\left(\|v\|_{L^\infty(\Rd\times I_{R/2})}+\left(\frac{R}{2}\right)^{2s}A\right)
\end{align*}
for any $(z,r)\in Q_{R/2}\setminus Q_\rho$ and some $\alpha=\alpha(d,s,\lambda,\Lambda,R_0)\in(0,2s)$ and $\tilde{C}=\tilde{C}(d,s,\lambda,\Lambda,R_0)>0$. For any $(x,t)\in Q_{R/4}$ with $d((x,t),(y,\tau))>\rho$, consider the point $(z,r)=(x-y,t-\tau)\in Q_{R/2}$. Notice that we can assume without loss of generality that $t<\tau$. Since
\begin{align*}
u(x,t)-u(y,\tau)=u(y+z,\tau+r)-u(y,\tau)=v(x,t)-v(0,0)
\end{align*}
and
\begin{align*}
d((x,t),(y,\tau))=\max\left\{|x-y|,(\tau-t)^\frac{1}{2s}\right\}=\max\left\{|z|,(-r)^\frac{1}{2s}\right\}=d((z,r),(0,0)),
\end{align*}
we obtain
\begin{align*}
\frac{|u(x,t)-u(y,\tau)|}{d((x,t),(y,\tau))^\alpha}=\frac{|v(z,r)-v(0,0)|}{d((z,r),(0,0))^\alpha}&\leq\tilde{C}\left(\frac{R}{2}\right)^{-\alpha}\left(\|v\|_{L^\infty(\Rd\times I_{R/2})}+\left(\frac{R}{2}\right)^{2s}A\right) \\
&\leq\frac{C}{R^\alpha}\left(\|u\|_{L^\infty(\Rd\times I_R)}+R^{2s}A\right)
\end{align*}
for $C=2^\alpha\tilde{C}=C(d,s,\lambda,\Lambda,R_0)>0$.
\end{proof}


\printbibliography


\end{document}